\documentclass{amsart}
\usepackage{amssymb, graphics, color, enumitem}
\usepackage[all]{xy}

 \usepackage{tikz}

 \usepackage[applemac]{inputenc}
\usepackage[cyr]{aeguill}

\usepackage[margin = 1in]{geometry}

\newtheorem{lemma}{Lemma}[section]
\newtheorem{proposition}[lemma]{Proposition}
\newtheorem{corollary}[lemma]{Corollary}
\newtheorem{theorem}[lemma]{Theorem}

\newtheorem{example}[lemma]{Example}
\newtheorem{definition}[lemma]{Definition}
\newtheorem{remark}[lemma]{Remark}

\newtheorem*{Acknowledgement}{Acknowledgements}

\newcommand\cf{cf\@. }

\newcommand\pa{ \partial}

\newcommand\bbC{\mathbb C}

\newcommand\bbN{\mathbb N}
\newcommand\bbP{\mathbb P}

\newcommand\bbR{\mathbb R}
\newcommand\bbS{\mathbb S}

\newcommand\bbZ{\mathbb Z}

\usepackage{color}

\newcommand\tX{\widetilde{X}}
\newcommand\tx{\widetilde{x}}

\newcommand\bX{\overline{X}}
\newcommand\bU{\overline{\mathcal{U}}}

\newcommand\CI{\mathcal{C}^{\infty}}

\newcommand\cC{\mathcal{C}}
\newcommand\cA{\mathcal{A}}
\newcommand\cE{\mathcal{E}}

\newcommand\cF{\mathcal{F}}
\newcommand\cL{\mathcal{L}}
\newcommand\cR{\mathcal{R}}
\newcommand\db{\overline{\pa}}

\newcommand\cV{\mathcal{V}}
\newcommand\cU{\mathcal{U}}

\newcommand\End{\operatorname{End}}

\newcommand\barf{\overline{f}}

\newcommand\phg{\operatorname{phg}}

\newcommand\Id{\operatorname{Id}}

\newcommand\tbeta{\widetilde{\beta}}

\renewcommand\sc{\operatorname{sc}}

\newcommand\cH{\mathcal{H}}

\newcommand\AC{\operatorname{AC}}

\newcommand{\bx}{\overline{x}}

\newcommand\QAC{\operatorname{QAC}}

\newcommand\Qb{\operatorname{Qb}}

\newcommand\bV{\overline{V}}
\newcommand\bW{\overline{W}}

\newcommand\tomega{\widetilde{\omega}}

\newcommand\tW{\widetilde{W}}

\newcommand\cX{\mathcal{X}}

\newcommand\cO{\mathcal{O}}
\newcommand\cW{\mathcal{W}}
\newcommand\dv{\operatorname{f}}

\begin{document}
\title[Complete Calabi-Yau metrics on $\bbC^n$]
{ New examples of complete Calabi-Yau metrics on $\bbC^n$ for $n\ge 3$}

\author{Ronan J.~Conlon}
\address{Department of Mathematics and Statistics, Florida International University, Miami, FL 33199, USA}
\email{rconlon@fiu.edu}

\author{Fr\'ed\'eric Rochon}
\address{Département de Mathématiques, Universit\'e du Qu\'ebec \`a Montr\'eal}
\email{rochon.frederic@uqam.ca}

\maketitle

\begin{abstract}
For each $n\ge 3$, we construct on $\bbC^n$ examples of complete Calabi-Yau metrics of Euclidean volume growth having a tangent cone at infinity with singular cross-section.  
\end{abstract}

\tableofcontents

\numberwithin{equation}{section}

\section{Introduction}

A complete K\"ahler manifold $M$ is Calabi-Yau if it is Ricci-flat and has a nowhere vanishing parallel holomorphic volume form $\Omega\in H^{0}(M,\,K_{M})$.
On compact K\"ahler manifolds, Calabi-Yau metrics are unique in their K\"ahler class, a statement which fails to hold true on non-compact Calabi-Yau manifolds, even modulo scaling and diffeomorphisms, as the Taub-NUT metric and flat metric on $\mathbb{C}^{2}$ demonstrate. However, the Taub-NUT metric has cubic volume growth whereas the flat metric has Euclidean volume growth. Thus, the question remains as to whether one can obtain uniqueness modulo scaling and diffeomorphisms for Calabi-Yau metrics of a fixed volume growth in each K\"ahler class.

In this article, we are concerned with Calabi-Yau manifolds of Euclidean volume growth. In this case, there is a tangent cone at infinity which is
unique when there exists a tangent cone with a smooth cross section \cite{Colding-Minicozzi} and which is likely to be unique in general. Examples of such manifolds that converge smoothly to a smooth Calabi-Yau cone at infinity at a rate of $O(r^{-\epsilon})$ -- so-called asymptotically conical Calabi-Yau manifolds -- have been constructed in \cite{vanC, vanC2011, Goto, CH2013, CH2015}, whereas examples with a unique singular tangent cone at infinity can be found in \cite{Joyce, Carron2011, CDR2016}. Simplifying the above problem by considering Calabi-Yau manifolds with Euclidean volume growth that have a unique tangent cone at infinity as in the aforementioned examples, one may ask whether the underlying complex manifold admits another Calabi-Yau metric of Euclidean volume growth in the same K\"ahler class as the initial metric but with a different tangent cone at infinity.

The simplest example to consider here is $\mathbb{C}^{n}$ endowed with the flat metric. On $\mathbb{C}^{2}$, Tian showed in \cite{Tian2006} that every Calabi-Yau metric of Euclidean volume growth has to be flat and conjectured that the same should hold true on $\mathbb{C}^{n}$ for all $n\geq 3$. On $\mathbb{C}^{3}$, a counterexample to this conjecture was recently found by Yang Li \cite{YangLi}. The purpose of this paper is to provide a counterexample to this conjecture on $\mathbb{C}^{n}$ for all $n\geq 3$. Before stating our main theorem, we introduce some preliminaries.

For $n\geq3$, let $F_{k}$ be a K\"ahler-Einstein Fano manifold, defined as the zero locus of a
homogeneous polynomial $P_{k}$ in $\mathbb{CP}^{n-1}$ of degree $k$ for $2\le k \le n-1$ (of which there are many examples). Then the affine
cone $V_{0,k}\subset \mathbb{C}^{n}$ over $F_{k}$, defined as the zero locus of $P_{k}$ considered as
an equation on $\mathbb{C}^{n}$, is a Calabi-Yau cone. With this, our main result takes the following form (see Theorem~\ref{cy.3} and Corollary~\ref{cy.4} for more details).
\begin{theorem}\label{main}
For each $n\ge 3$ and for each K\"ahler-Einstein Fano manifold $F_{k}$ as above, there exists a complete Calabi-Yau metric on $\bbC^n$ with Euclidean volume growth with tangent cone at infinity $\mathbb{C}\times V_{0,\,k}$. In particular, these metrics are not isometric to the Euclidean metric on $\bbC^n$.
\end{theorem}

\begin{remark}
After our preprint was posted, G\'abor Székelyhidi posted a preprint \cite{Szekelyhidi} where he recovers our results using different methods. 
\end{remark}

Since the tangent cones of the Calabi-Yau metrics in Theorem \ref{main} all have singular cross-section, the question remains as to whether $\mathbb{C}^{n}$ (or indeed, any asymptotically conical Calabi-Yau manifold) admits another Calabi-Yau metric with Euclidean volume growth with a different tangent cone at infinity also having a smooth cross-section.

\subsection{Outline of the proof of Theorem \ref{main}}

The general philosophy for constructing Calabi-Yau manifolds with a prescribed tangent cone at infinity is to take an affine $\bbC^*$-equivariant deformation of the cone preserving the asymptotics at infinity followed by a K\"ahler crepant resolution, then construct an asymptotically Calabi-Yau metric on the resulting K\"ahler manifold, and finally perturb this metric to a Calabi-Yau metric by solving the complex Monge-Amp\`ere equation. Here we work on a deformation of the cone $\mathbb{C}\times V_{0,\,k}$ that yields $\mathbb{C}^{n}$ and our construction of an asymptotic Calabi-Yau metric is strongly inspired by an ansatz of Hans-Joachim Hein and Aaron Naber in their unpublished work \cite{Naber}.

Let $P_k(z_1,\ldots,z_n)$ be a homogeneous polynomial of degree $k$ for $2\le k \le n-1$ such that the hypersurface
$$
       F_k:= \left\{ [0:z_1: \ldots : z_n] \in \bbC\bbP^{n-1}\subset \bbC\bbP^{n} \; |  \; P_k(z_1,\ldots,z_n)=0 \right\}
$$
is smooth and admits a Kähler-Einstein metric.  By \cite{Tian1987, Nadel, Arezzo}, we can take for instance $F_k$ to be a Fermat hypersurface or a smooth hypersurface sufficiently close to a Fermat hypersurface.  On $\bbC^{n}$, we consider the singular affine variety
$$
     V_{0,k}:= \{ (z_1,\ldots,z_n)\in \bbC^n \; | \; P_k(z_1,\ldots,z_n) =0\},
$$
which by the Calabi ansatz admits a Calabi-Yau cone metric $g_{V_{0,k}}$.  In terms of this Calabi-Yau cone, the prescribed tangent cone at infinity of our examples is the Cartesian product $( V_{0,k}\times \bbC, g_{V_{0,k}}\times g_{\bbC})$, where $g_{\bbC}$ is the Euclidean metric on $\bbC$.   We then consider the smoothing $W_{\epsilon,k}$ of $V_{0,k}\times \bbC$ given by the equation
$$
     P_k(z_1,\ldots,z_n)=\epsilon z_{n+1}
$$
for $\epsilon\ne 0$.  This smoothing, being the graph of the polynomial $\epsilon^{-1}P_k(z_1,\ldots,z_n)$, is biholomorphic to $\bbC^n$.  Our strategy then comprises first in constructing examples of Kähler metrics on $W_{\epsilon,k}$ that are  modelled on $g_{\bbC}\times g_{V_{0,k}}$ at infinity. To do this, we compactify $W_{\epsilon,k}$ in a suitable way by a manifold with corners $\cW_{\epsilon,k}$ having two boundary hypersurfaces $\cH_1$ and $\cH_2$ describing the two types of behavior of the metric at infinity.  At $\cH_2$, the metric behaves like an asymptotically conical metric ($\AC$-metric for short), whereas near $\cH_1$, it is modelled on a \textbf{warped} product of $\AC$-metrics,
$$
   g_{\bbC} + |z_{n+1}|^{\frac{2(n-k)}{k(n-1)}} g_{V_{\epsilon,k}},
$$
where $V_{\epsilon,k}$ is the smoothing of $V_{0,k}$ defined by the equation
$$
   P_{k}(z_1,\ldots,z_n)=\epsilon
$$
and $g_{V_{\epsilon,k}}$ can be chosen to be a Calabi-Yau AC-metric on $V_{\epsilon,k}$.  Notice that if we instead consider a \textbf{Cartesian} product of $\AC$-metrics, then this would correspond to a special case of the quasi-asymptotically conical metrics ($\QAC$-metrics for short) introduced by Degeratu and Mazzeo \cite{DM2014}.  For this reason, we call the metrics we consider in the present paper \textbf{warped} $\QAC$-metrics.   One other common point with $\QAC$-metrics to highlight is that a warped $\QAC$-metric is conformal to  a $\Qb$-metric as defined in \cite[Definition~{1.29}]{CDR2016}, although with a different conformal factor, a useful fact that yields a simple coordinate free definition of warped $\QAC$-metrics; see Definition~\ref{qb.11} below.

For this class of metrics, we first construct examples that are Kähler with Ricci potential decaying at infinity.  To apply the work of Tian-Yau \cite{Tian-Yau1991} and get a Calabi-Yau metric, we need however to improve the decay of the Ricci potential at infinity.  To do this, we need to derive good mapping properties of the Laplacian for these metrics on suitable weighted H\"older spaces, more precisely  the same Hölder spaces as in \cite{Joyce, DM2014, CDR2016}, namely those associated to a $\Qb$-metric, but with different weights.  As in \cite{DM2014} for $\QAC$-metrics, we do this by deriving estimates for the heat kernel and the Green's function using the work of Grigor'yan and Saloff-Coste \cite{GS2005}.  With these mapping properties, we can then improve the decay of the Ricci potential at infinity, which ultimately allows us to solve a corresponding complex Monge-Ampère equation and obtain our new examples of complete Calabi-Yau metrics.

\subsection{Overview} The paper is organized as follows.  In \S~\ref{ac.0}, we recall the results that we need concerning Calabi-Yau $\AC$-metrics.  In \S~\ref{w.0}, we introduce the compactification of $W_{\epsilon,k}$ by a manifold with corners, the class of warped $\QAC$-metrics and the various functional spaces that we require.  In \S~\ref{acy.0}, we construct examples of Kähler warped $\QAC$-metrics with Ricci potential decaying at infinity.  In \S~\ref{mp.0}, we show that the Laplacian of a warped $\QAC$-metric is an isomorphism when acting on suitable weighted Hölder spaces.  In \S~\ref{cy.0}, we use these properties to solve a complex Monge-Ampère equation and obtain our new examples of Calabi-Yau metrics on $\bbC^n$ for $n\ge 3$.  Finally, in \S~\ref{k2.0}, we treat the case $k=2$ and $n\ge 4$ separately as it requires further work.  In particular, for this last part, we need to invoke some results about the mapping properties of the Laplacian of an incomplete edge metric which are discussed separately in Appendix A.

\begin{Acknowledgement}
Our construction of an asymptotically Calabi-Yau metric originates from an ansatz of Hans-Joachim Hein and Aaron Naber in their unpublished work \cite{Naber}.  The authors wish to thank Hans-Joachim for many helpful conversations and G\'abor Székelyhidi for helpful email exchanges.  The second author was supported by NSERC and a Canada Research Chair.
\end{Acknowledgement}

\section{Asymptotically conical Calabi-Yau metrics via smoothing} \label{ac.0}

Let $F_k\subset \bbC\bbP^{n-1}\subset \bbC\bbP^n$ be a smooth hypersurface of degree $k$ in $\bbC \bbP^{n-1}$ given by
\begin{equation}
  F_k= \{ [0:z_1:\ldots: z_n]\in \bbC\bbP^{n-1}\subset \bbC\bbP^n \; | \; P_k(z_1,\ldots,z_n)=0 \},
\label{ac.3}\end{equation}
where $P_k$ is a homogeneous polynomial of degree $k$ such that
\begin{equation}
  2 \le k \le n-1.
\label{ac.6}\end{equation}
 Using the adjunction formula for $F_k\subset \bbC\bbP^{n-1}$, we see that $K_{F_k}= \cO(-n+k)$, so that $c_1(F_k)>0$ and $F_k$ is Fano.  Assume furthermore that    
 $F_k$ admits a Kähler-Einstein metric $g_{F_k}$.  
 \begin{example}
 We may take 
 $$
 P_k(z_1,\ldots,z_n)= \sum_{i=1}^n z_i^k,
 $$
 since then $F_k$ is a Fermat hypersurface and we know from the work of Tian \cite{Tian1987} if $k=n-1,n-2$, the work of Nadel \cite{Nadel} if $\frac{n-1}{2}\le k\le n-3$ and by \cite{Arezzo} for all the remaining $k$ satisfying \eqref{ac.6} that $F_k$ admits a Kähler Einstein metric.  More generally, since being Kähler-Einstein is an open condition on the moduli space of smooth hypersurfaces of degree $k$, it suffices to take $F_k$ to be sufficiently close to a Fermat hypersurface.   
 \label{ex.1}\end{example}
 \begin{example}
 We can take a smooth hypersurface $F_k$ of the form \eqref{ac.3} with 
 $$
      P_k(z_1,\ldots,z_n)= \left(\sum_{i=1}^{\ell-1} z_i^k \right) + Q_k(z_{\ell},\ldots,z_n),
 $$
 where $Q_k$ is a homogeneous polynomial of degree $k$ and $\ell\in \bbN$ is chosen so that $\ell>n+1-k$, since then we know from \cite[Proposition~3.1]{Arezzo} that $F_k$ admits a Kähler Einstein metric.  Again, since being Kähler-Einstein is an open condition, it suffices to take $F_k$ sufficiently close to such a hypersurface.  
 \label{ex.2}\end{example}
 
Consider then the  affine variety $V_{0,k}$ in $\bbC^n$ defined by 
\begin{equation}
      V_{0,k}:= \{  (z_1,\ldots, z_n)\in \bbC^n \; | \; P_k(z_1,\ldots,z_n)=0\}.
\label{ac.1}\end{equation}
It is in fact quasi-projective since $V_{0,k}= \bV_{0,k}\setminus F_k$ where $\bbP V_{0,k}$ is the projective variety given by the closure of $V_{0,k}$ in $\bbC\bbP^n$,
\begin{equation}
 \bbP V_{0,k}= \{ [z_0:z_1:\ldots:z_n]\in \bbC\bbP^{n} \; | \; P_k(z_1,\ldots,z_n) =0\}.
\label{ac.2}\end{equation}
Since the normal bundle $N_{F_k}$ of $F_k$ in $\bbP V_{0,k}$ is $\cO(1)$, this means that
\begin{equation}
      K_{F_k}= N_{F_k}^{\otimes (k-n)}. 
\label{ac.4}\end{equation}
Moreover, using the natural $\bbC^*$-action on $V_{0,k}$, notice that there is a canonical identification 
\begin{equation}
         V_{0,k}\setminus \{0\} = N_{F_k}^*\setminus \{F_k\}.
\label{ac.5}\end{equation}
Using the Kähler-Einstein metric $g_{F_k}$ on $F_k$, we can therefore apply the Calabi ansatz as in \cite{Calabi, LeBrun} to construct a Calabi-Yau cone metric $g_{V_{0,k}}$ on $V_{0,k}$ with Kähler form given explicitly by
$$
    \omega_{V_{0,k}}= \frac{\sqrt{-1}}2 \pa \db \| \cdot \|^{\frac{2(n-k)}{n-1}}_{N^*_{F_k}}, 
$$
where the norm on $N^*_{F_k}$, which is a $(n-k)$th root of $K_{F_k}$, is the one induced by the Kähler-Einstein metric $g_{F_k}$. In particular, the corresponding radial function is given by 
\begin{equation}
r:= \| \cdot \|^{\frac{(n-k)}{n-1}}_{N^*_{F_k}} \quad \Longrightarrow \quad r\asymp |z|^{\frac{n-k}{n-1}} \quad \mbox{on}\; V_{0,k},
\label{ac.6b}\end{equation}
where $|z|^2= \sum_{i=1}^n |z_i|^2$ and the notation $f_1\asymp f_2$ means that there exist positive constants $c$ and $C$ such that $cf_2\le f_1\le Cf_2$. The fact that $g_{V_{0,k}}$ is Calabi-Yau means more specifically that there is a constant $c_{n-1}$ such that 
$$
    \omega_{0,k}^{n-1}= c_{n-1}\Omega^{n-1}_0\wedge \overline{\Omega}^{n-1}_0,
$$ 
where $\Omega_0^{n-1}$ is the holomorphic volume form on $V_{0,k}$ defined implicitly by 
$$
       \left. dz_1\wedge\ldots \wedge dz_n \right|_{V_{0,k}}= \left. \Omega^{n-1}_0\wedge d\left( P_k(z_1,\ldots,z_n) \right)\right|_{V_{0,k}}.
$$
Now, for $\epsilon\ne 0$, we can deform $V_{0,k}$ to the smooth affine hypersurface $V_{\epsilon,k}\subset \bbC^n$ defined by the equation
\begin{equation}
         P_k(z_1,\ldots,z_n)= \epsilon.
\label{ac.7}\end{equation}
Again, $V_{\epsilon,k}$ is a quasi-projective variety, since $V_{\epsilon,k}= \bbP V_{\epsilon,k}\setminus F_k$ where 
$$
   \bbP V_{\epsilon,k}:= \{ [z_0: z_1:\ldots:z_n]\in \bbC\bbP^n \;| \; P_k(z_1,\ldots,z_n)= \epsilon z_0^k\}.
$$
By the adjunction formula, $K_{\bbP V_{\epsilon,k}}= \cO(-n-1+k)$, so that if we denote also by $N_{F_k}$ the normal bundle of $F_k$ in $\bbP V_{\epsilon,k}$, we have that  $N_{F_k}= \cO(1)$ and that 
$$
       -K_{\bbP V_{\epsilon,k}}= N_{F_k}^{\otimes n+1-k}.
$$
On $\bbP V_{\epsilon,k}$, there is also a natural holomorphic volume form $\Omega^{n-1}_{\epsilon}$ given implicitly by
$$
 \left. dz_1\wedge\ldots \wedge dz_n \right|_{V_{\epsilon,k}}= \left. \Omega^{n-1}_{\epsilon}\wedge d\left( P_k(z_1,\ldots,z_n)-\epsilon \right)\right|_{V_{\epsilon,k}}.
$$
This is a setting where the result of Tian-Yau \cite{Tian-Yau1991} applies.  We will use the result in a version adapted to $\AC$-metrics as  in \cite{Joyce, vanC, vanC2011, Goto,CH2013,CH2015}.  
  
Let $\overline{\bbC^n}$ denote the radial compactification of $\bbC^n$.  Notice that this radial compactification can be obtained from $\bbC\bbP^{n}$ by blowing up the divisor at infinity $\bbC\bbP^{n-1}$ in the sense of Melrose \cite{MelroseAPS}, so that there is a blow-down map $\beta_n:\overline{\bbC^{n}}\to \bbC\bbP^n$.  Let $\bV_{\epsilon,k}:= \beta_n^{-1}(\bbP V_{\epsilon,k})$ be the closure of $V_{\epsilon,k}$ in $\overline{\bbC^n}$.  Let $\bU$  be an open neighborhood of $\beta_n^{-1}(F_k)$   in $\overline{\bbC^n}$ not intersecting the origin in $\bbC^n$ and set $\bU_{\epsilon}= \bU\cap \bV_{\epsilon,k}$. Since for $\epsilon\ne 0$, $\bbP V_{\epsilon,k}$ and $\bbP V_{0,k}$ are tangent to order $k$ on their intersection on $F_k$, we see, \cf \cite[Proposition~1.3(2) and Remark~5.7]{ChiLi}, that there exists a smooth diffeomorphism $\varphi_{\epsilon}: \bU_{0} \to \bU_{\epsilon}$ such that if $J_{\epsilon}$ denotes the complex structure of $\bbP V_{\epsilon,k}$ and its lift to ${}^{b}T\bV_{\epsilon,k}$, then
\begin{equation}
  \varphi^*_{\epsilon}J_{\epsilon}- J_0\in |z|^{-k}\CI(\bU_0; \End({}^{b}T^*\bU_0))\subset |z|^{-k}\CI_b(\cU_0; \End({}^{b}T\bU_0)),
\label{ac.8}\end{equation}
where $\cU_0= \bU_0\cap V_{0,k}$,${}^{b}T\bU_0$ is the  $b$-tangent bundle as introduced in \cite{MelroseAPS} and $\CI_b(\cU_0; \End({}^{b}T^*\bU_0))$ is the space of smooth bounded conormal sections as defined in \cite{CMR2015}.  Similarly, we have that 
\begin{equation}
   \log\left(\frac{\varphi_{\epsilon}^* \Omega^{n-1}_{\epsilon}}{\Omega^{n-1}_0}\right)\in |z|^{-k}\CI(\bU_0)\subset |z|^{-k}\CI_b(\cU_0).
\label{ac.9}\end{equation}
Given a manifold with corners $M$, $F\to M$ a smooth Euclidean vector bundle and $\cE$ an index family associated to the boundary hypersurfaces of $M$, recall from \cite{Melrose1992, MelroseAPS} that $\cA_{\phg}^{\cE}(M;F)$ denotes the space of bounded smooth sections of $F$ on $M\setminus \pa M$ that admit a polyhomogeneous expansion at each boundary hypersurface with respect to the index family $\cE$.  In this paper, we will mostly be concerned with such sections that are bounded, regardless of which index family $\cE$ describes the polyhomogeneous expansion.  Thus, we will denote by $\cA_{\phg}(M;E)$ the space of bounded smooth sections on $M\setminus \pa M$ that admit a polyhomogeneous expansion at each boundary hypersurface with respect to some index family.  In particular, we have that $\CI(M;E)\subset \cA_{\phg}(M;E)$.
\begin{remark}
Since we are using the $b$-tangent bundle instead of the scattering tangent bundle of \cite{MelroseGST} given formally by ${}^{\sc}T\bU_0= \frac{1}{r}{}^b T\bU_0$, we have that
$$
 \varphi_{\epsilon}^* \Omega^{n-1}_{\epsilon}, \Omega^{n-1}_0\in |z|^{n-k}\CI(\bU_0; \Lambda^{n-1}_{\bbC}\left({}^{b}T^*\bU_0\otimes \bbC\right))=r^{n-1}\CI(\bU_0; \Lambda^{n-1}_{\bbC}\left({}^{b}T^*\bU_0\otimes \bbC\right)).
$$
Similarly, 
$$
      g_{V_{0,k}}\in r^2\CI_b(\bU_0; S^2({}^bT^*\bU_0)).
$$      
\label{ac.9b}\end{remark}

\begin{theorem}
For $2\le k\le n-1$, the affine variety $V_{\epsilon,k}$ admits a complete Calabi-Yau metric $g_{V_{\epsilon,k}}$ with Kähler form $\omega_{V_{\epsilon,k}}=\frac{\sqrt{-1}}2\pa\db v_{\epsilon}$ such that 
\begin{equation}
\begin{gathered}
        \varphi_{\epsilon}^*v_{\epsilon}-r^2\in |z|^{(-\mu+2)\frac{(n-k)}{n-1}}\cA_{\phg}(\overline{\cU_0})\subset |z|^{(-\mu+2)\frac{(n-k)}{n-1}}\CI_b(\cU_0), \\ 
         \varphi_{\epsilon}^{*} g_{V_{\epsilon,k}}- g_{V_{0,k}}\in |z|^{(-\mu+2)\frac{n-k}{n-1}} \cA_{\phg}(\overline{\cU}_0; {}^{b}T^*\cU_0\otimes {}^{b}T^*\cU_0)\subset |z|^{(-\mu+2)\frac{n-k}{n-1}} \CI_b(\cU_0; {}^{b}T^*\cU_0\otimes {}^{b}T^*\cU_0),  \\  
         \omega_{\epsilon,k}^{n-1}= c_{n-1}\Omega^{n-1}_{\epsilon}\wedge \overline{\Omega}^{n-1}_{\epsilon},
\end{gathered}         
\label{ac.9b}\end{equation}         
where 
\begin{equation}
   \mu= \left\{ \begin{array}{ll} k\frac{n-1}{n-k} & \mbox{if}\; k<\frac{2n}3, \\ 2n-2 & \mbox{if} \; k> \frac{2n}3,  \\ 2n-2-\delta \; & \mbox{if}\; k=\frac{2n}3 \end{array}\right.
\label{ac.10}\end{equation}
and $\delta>0$ can be taken arbitrarily small.  
\label{ac.11}\end{theorem}
\begin{proof}
Recall the definition of the radial function $r$ in \eqref{ac.6b}.  Clearly, restricting the Euclidean metric on $\bbC^n$ to $V_{\epsilon,k}$ and using a standard convexity argument as in for instance \cite[Lemma~4.1]{CDR2016} or the proof of Corollary~\ref{w.31} below, there exists a potential $v\in \CI(V_{0,k})$ such that $\varphi_{\epsilon}^*v=r^2$ outside some compact set and $\frac{\sqrt{-1}}2 \pa \db v$ is the  Kähler form of an $\AC$-metric on $V_{\epsilon,k}$.  Thanks to \eqref{ac.8} and \eqref{ac.9}, the Ricci potential
$$
 r_v:= \log \left( \frac{(\frac{\sqrt{-1}}{2}\pa\db v)^{n-1}}{c_{n-1}\Omega_{\epsilon}^{n-1}\wedge\overline{\Omega}^{n-1}_{\epsilon}} \right)
$$
is in $|z|^{-k}\CI_b(V_{\epsilon,k})= r^{-k\frac{n-1}{n-k}}\CI_b(V_{\epsilon,k})$.  Hence, by \cite[Theorem~2.1]{CH2013}, the complex Monge-Ampère equation
$$
         \log\left(  \frac{(\frac{\sqrt{-1}}{2}\pa\db v+ \frac{\sqrt{-1}}{2}\pa\db w)^{n-1}}{c_{n-1}\Omega_{\epsilon}^{n-1}\wedge\overline{\Omega}^{n-1}_{\epsilon}}\right)=-r_v
$$
has a unique solution 
\begin{equation}
w\in r^{-\mu+2}\CI_b(V_{\epsilon,k}).
\label{ac.12}\end{equation} 
Hence, it suffices to take $v_{\epsilon}= v+w$.  Notice moreover that we know from \cite[\S~5]{CMR2015} that $v_{\epsilon}$ admits a polyhomogeneous expansion.  

\end{proof}

\begin{remark}
The equation \eqref{ac.11} shows that the metric $g_{V_{\epsilon,k}}$ is asymptotic to the cone metric $g_{V_{0,k}}$ with rate $r^{-\mu}$, so that it is an $\AC$-metric  with tangent cone $V_{0,k}$, \cf \cite[Definition~1.11]{CH2013}.
\label{ac.13}\end{remark}
\begin{remark}
When $k=\frac{2n}3$, thanks to \cite[\S~5]{CMR2015}, we can be more specific and say that 
$$
 \varphi_{\epsilon}^*v_{\epsilon}-r^2- f |z|^{(4-2n)\frac{(n-k)}{n-1}}\log|z|\in |z|^{(4-2n)\frac{(n-k)}{n-1}}\cA_{\phg}(\overline{\cU_0})
$$
for some function $f\in \CI(\bU_0)$.  That is, the $\delta>0$ in \eqref{ac.10} is due to the possible presence of a  logarithmic term in the polyhomogeneous expansion  of $\varphi_{\epsilon}^*v_{\epsilon}-r^2$.
\label{log.1}\end{remark}

\section{Warped $\QAC$-metrics on $\bbC^n$} \label{w.0}

On $\bbC^{n+1}$, consider the singular affine hypersurface  $W_{0,k}$ of degree $k$ defined by 
\begin{equation}
           W_{0,k}:= \{(z_1,\ldots,z_n,z_{n+1})\in \bbC^{n+1}\; |\;  P_k(z_1,\ldots,z_n)=0 \},
\label{w.1}\end{equation}
where $P_k$ is the homogeneous polynomial of \S~\ref{ac.0}. 
Since the defining equation of $W_{0,k}$ does not involve the variable $z_{n+1}$, $W_{0,k}$ is canonically isomorphic to the Cartesian product $V_{0,k}\times \bbC$.  Hence, under this identification, we can consider the product metric
\begin{equation}
  g_{W_{0,k}}:= g_{V_{0,k}} + g_{\bbC},
\label{w.2}\end{equation}
where $g_{\bbC}$ is the canonical Euclidean metric on $\bbC$.  Obviously $g_{W_{0,k}}$ and $W_{0,k}$ are singular at $z=0$, where $z=(z_1,\ldots,z_n)$ involves only the first $n$ variables.  Otherwise, away from $z=0$, $g_{W_{0,k}}$ is a Cartesian product of $\AC$-metrics, and as such, is an example of  a $\QAC$-metric; see \cite[\S2.3.5]{DM2014} or \cite[Example~1.22]{CDR2016}.

To remove the singularities, we can consider, for $\epsilon\ne 0$, the smooth deformation $W_{\epsilon,k}$ defined by 
\begin{equation}
  W_{\epsilon,k}:= \{ (z_1,\ldots, z_{n+1})\in \bbC^{n+1}\; | \; P_k(z_1,\ldots,z_n)= \epsilon z_{n+1}\}.
\label{w.3}\end{equation}
Clearly, this affine hypersurface is biholomorphic to the graph of the holomorphic function 
$$
      \begin{array}{llcl}
         f: &\bbC^n & \to & \bbC \\
            & z &\mapsto & \epsilon^{-1}P_k(z_1,\ldots,z_n),
      \end{array}
$$
which is itself biholomorphic to $\bbC^n$.  Obviously then, a natural Calabi-Yau metric to consider on $W_{\epsilon,k}$ is the Euclidean metric coming from the identification $W_{\epsilon,k}\cong \bbC^n$.  However, using the fact that $W_{\epsilon,k}$ is close to $W_{0,k}$ for $\epsilon$ small suggests that on $W_{\epsilon,k}$, there might be as well a Calabi-Yau metric on $W_{\epsilon,k}$ which at infinity behaves more like $g_{W_{0,k}}$.  

To make this statement rigorous, recall first that the radial function $r$ of $g_{V_{0,k}}$ behaves like $|z|^{\frac{n-k}{n-1}}$ at infinity, while the radial function on $g_{\bbC}$ is $|z_{n+1}|$.  Thus, with respect to the metric $g_{W_{0,k}}$, this means that the natural compactification of $\bbC^{n+1}$ we should consider is not the usual radial compactification, but instead the weighted radial compactification $\overline{\bbC^{n+1}_a}$ associated to the     
weighted $\bbC^*$-action 
\begin{equation}
   \lambda \cdot (z_1,\ldots, z_n,z_{n+1}) = (\lambda^{n-1}z_1, \ldots, \lambda^{n-1}z_n, \lambda^{n-k}z_{n+1})
\label{w.4}\end{equation}
corresponding to the weight $a:=(n-1,\ldots,n-1,n-k)\in \bbN^{n+1}$.  To give the precise definition of the manifold with boundary $\overline{\bbC^{n+1}_a}$, we first need  to consider the weighted sphere
\begin{equation}
   \bbS^{2n+1}_a := \left\{ (z_1,\ldots,z_n,z_{n+1})\in \bbC^{n+1} \; | \; \left(\sum_{i=1}^n |z_i|^{2(n-k)}\right)+ |z_{n+1}|^{2(n-1)}=1\right\}.
\label{w.5}\end{equation}
Then $\overline{\bbC^{n+1}_a}$ is obtained by gluing $\bbS^{2n+1}_a\times [0,\infty)$ and $\bbC^{n+1}$ via the map    
\begin{equation}
\begin{array}{lccl}
   \nu: & \bbS^{2n+1}_a\times (0,\infty) & \to & \bbC^{n+1} \\
          & (\omega,\xi) & \mapsto & \left( \frac{\omega_1}{\xi^{n-1}},\ldots, \frac{\omega_n}{\xi^{n-1}}, \frac{\omega_{n+1}}{\xi^{n-k}}\right),
          \end{array}
\label{w.6}\end{equation}
which extends to give a tubular neighborhood $\nu: \bbS^{2n+1}_a\times [0,\infty)\to \overline{\bbC^{n+1}_a}$ of $\pa \overline{\bbC^{n+1}_a}= \bbS^{2n+1}_a$ in $\overline{\bbC^{n+1}_a}$.  In terms of the coordinates $(\omega,\xi)$ induced by this tubular neighborhood, we see that the defining equation of $W_{\epsilon,k}$ takes the form
\begin{equation}
  P_k(\omega_1,\ldots,\omega_n)= \epsilon \omega_{n+1} \xi^{n(k-1)}, \quad (\omega_1,\ldots, \omega_{n+1}) \in \bbS^{2n+1}_a\in \bbC^{n+1}.
\label{w.7}\end{equation}
Hence, the closure $\bW_{\epsilon,k}$ of $W_{\epsilon,k}$ in $\overline{\bbC^{n+1}_a}$ is obtained by considering the case $\xi=0$ in \eqref{w.7}, so that the 
defining equation of 
\begin{equation}
\pa\bW_{\epsilon,k}:=\bW_{\epsilon,k}\cap \pa \overline{\bbC^{n+1}_a}
\label{w.7b}\end{equation} 
is the zero locus of the equation
\begin{equation}
    P_k(\omega_1,\ldots,\omega_n)=0 \quad \mbox{in} \quad \bbS^{2n+1}_a.
\label{w.8}\end{equation}
In particular, $\pa\bW_{\epsilon,k}=\pa\bW_{0,k}$ is independent of $\epsilon$ and  is singular on the circle defined by the equation $|\omega_{n+1}|=1$ on $\bbS^{2n+1}_a$.  

Before discussing further this singularity, let us describe another natural compactification which behaves better with respect to the complex structure.  Notice first that the weighted $\bbC^*$-action  \eqref{w.4} restricts to a weighted $\bbS^1$-action on $\bbS^{2n+1}_a$.  This action is not faithful in general.  If we denote by $\ell$  the greatest common divisor of $n-1$ and $n-k$, then the subgroup $\bbZ_{\ell}\subset \bbS^1$ acts trivially on $\bbS^{2n+1}_a$.  Thus, we have a well-defined $\bbS^1/\bbZ_{\ell}$-action which is itself faithful.  The quotient of $\bbS^{2n+1}_a$ by this action is in fact the weighted projective space $\bbC\bbP^{n}_a$ associated to the weighted $\bbC^*$-action \eqref{w.4}.  It is a complex orbifold and the quotient map
\begin{equation}
 \xymatrix{  \bbS^1/\bbZ_{\ell} \ar[r] & \bbS^{2n+1}_a\ar[d]^{q} \\ & \bbC\bbP^n_a
 }          
\label{w.8a}\end{equation}
induces the structure of an $\bbS^1/\bbZ_{\ell}$-orbibundle over $\bbC\bbP^n_a$.  At the level of sets, the other natural compactification we want to consider is to add  $\bbC\bbP^{n}_a$ at infinity instead of $\bbS^{2n+1}_a$.  In complex geometric terms, what we want to consider is a weighted projective compactification compatible with the $\bbC^*$-action \eqref{w.4}.   More precisely, the compactification will be the weighted projective space $\bbC\bbP^{n+1}_{(1,a)}$ with weight 
$$
     (1,a)=(1,n-1,\ldots,n-1,n-k)\in \bbN^{n+2}.
$$
Recall for instance from \cite{Joyce} that this means that $\bbC\bbP^{n+1}_{(1,a)}$ is the quotient of $\bbC^{n+2}\setminus \{0\}$ with respect to the $\bbC^*$-action   
\begin{equation}
  \lambda\cdot (z_0,z_1,\ldots, z_n, z_{n+1}) = (\lambda z_0, \lambda^{n-1}z_1,\ldots, \lambda^{n-1}z_n, \lambda^{n-k}z_{n+1}),
\label{w.9}\end{equation}
and that, as for the usual projective space, we denote by $[z_0:z_1:\ldots: z_n: z_{n+1}]$ the point of $\bbC\bbP^{n+1}_{(1,a)}$ corresponding to the orbit of $(z_0,\ldots,z_{n+1})\in \bbC^{n+2}\setminus \{0\}$.  Notice that $\bbC^{n+1}$ is naturally included in $\bbC\bbP^{n+1}_{(1,a)}$ via the map 
$$
       \bbC^{n+1} \ni (z_1,\ldots, z_{n+1})\mapsto [1:z_1:\ldots : z_{n+1}]\in \bbC\bbP^{n+1}_{(1,a)}.
$$
Similarly, $\bbC\bbP^n_a$ is naturally included in $\bbC\bbP^{n+1}_{(1,a)}$ via the map
$$
\bbC\bbP^{n}_a \ni [z_1:\ldots: z_{n+1}]\mapsto [0:z_1:\ldots : z_{n+1}]\in \bbC\bbP^{n+1}_{(1,a)}.$$ 
Clearly, the singularities of $\bbC\bbP^{n+1}_{(1,a)}$ all lie on this $\bbC\bbP^n_a$ which we can   think of as located at infinity.  
If $\ell=1$, that is, if $n-1$ and $n-k$ are coprime, the orbifold singularities of $\bbC\bbP^{n+1}_{(1,a)}$ are given by
\begin{equation}
     \{[0: z_1:\ldots: z_n:0] \in \bbC\bbP^{n+1}_{(1,a)} \; | \; (z_1,\ldots,z_n)\in \bbC^n\setminus\{0\}\},
\label{w.10}\end{equation}
as well as the point $[0:\ldots:0:1]\in \bbC\bbP^{n+1}_{(1,a)}$ if $n-k>1$.  If instead $\ell>1$, then $\bbC\bbP^n_a$ itself is singular within $\bbC\bbP^{n+1}_{(1,a)}$ with singularity locally modelled on 
$$
      \bbC^n\times (\bbC/\bbZ_{\ell}) \quad \mbox{with $\bbZ_{\ell}$-action induced by multiplication}, 
$$  
except near  \eqref{w.10}, where it is modelled on 
$$
    \bbC^{n-1}\times (\bbC^2/\bbZ_{n-1}) \quad \mbox{with action generated by}\; e^{\frac{2\pi i}{n-1}}\cdot(z_0,z_{n+1})= (e^{\frac{2\pi i}{n-1}}z_{0},e^{\frac{2\pi i(n-k)}{n-1}}z_{n+1}),
$$
and near the point $[0:\ldots:0 :1]$, where it is modelled on 
$$
   \bbC^{n+1}/\bbZ_{n-k} \quad\mbox{with action generated by}\; e^{\frac{2\pi i}{n-k}}\cdot (z_0,\ldots,z_{n})= (e^{\frac{2\pi i}{n-k}}z_0,e^{\frac{2\pi i(n-1)}{n-k}}z_1,\ldots, e^{\frac{2\pi i(n-1)}{n-k}}z_n).
$$
Clearly, there is a natural map 
$$
   \beta: \overline{\bbC^{n+1}_a}\to \bbC\bbP^{n+1}_{(1,a)}
$$
which restricts to the identity on $\bbC^{n+1}$ and is given by $\beta(\omega,\xi)=[\frac{\omega_1}{\xi^{n-1}}:\ldots : \frac{\omega_n}{\xi^{n-1}}: \frac{\omega_{n+1}}{\xi^{n-k}}]$ in the coordinates $(\omega,\xi)$ of \eqref{w.6}.  In particular, $\left.\beta\right|_{\pa \overline{\bbC^{n+1}_a}}$ is just the quotient map \eqref{w.8a}, so forgetting about the orbifold singularities of $\bbC\bbP^n_a$ and $\bbC\bbP^{n+1}_{(1,a)}$, we can think of $\beta$ as a blow-down map with $\overline{\bbC^{n+1}_a}$ obtained from $\bbC\bbP^{n+1}_{(1,a)}$ by `blowing up $\bbC\bbP^n_a$ in the sense of Melrose' \cite{MelroseAPS}.  

 In $\bbC\bbP^{n+1}_{(1,a)}$, the closure $\bbP W_{\epsilon,k}$ of $W_{\epsilon,k}$ is defined by the equation
$$
        P_k(z_1,\ldots, z_n)= \epsilon z_{n+1}z_0^{n(k-1)} \quad \mbox{for} \; [z_0:\dots: z_{n+1}]\in \bbC\bbP^{n+1}_{(1,a)}.
$$   
Notice in particular that 
$$
   \bbP W_{\epsilon,k}\cap \bbC\bbP^n_a = \{ [0:z_1:\ldots: z_{n+1}]\in \bbC\bbP^n_a\; | \; P_k(z_1,\ldots,z_n)=0\}
$$
is independent of $\epsilon$ and is singular at the point $[0:\ldots: 0:1]$.

Let $S_1\subset \pa \overline{\bbC^{n+1}_a}$ be the circle along which  $\bW_{\epsilon,k}$ is singular.  To resolve this singularity, we need to blow up $S_1$ in $\overline{\bbC^{n+1}_a}$, but making a different choice of boundary defining function.  More precisely, let $\bX$ be the manifold with boundary which is $\overline{\bbC^{n+1}_a}$ as a topological space, but with space of smooth functions those of $\overline{\bbC^{n+1}_a}$ having a Taylor expansion at $\pa \overline{X}$ involving only integer powers of $\bx:=\xi^{\frac{n(k-1)}{k}}$ instead of $\xi$.  With this understood, we then consider the space $\tX:= [\bX;S_1]$ obtained by blowing up $S_1$ in the sense of Melrose \cite{MelroseAPS} with blow-down map
\begin{equation}
    \tbeta: \tX\to \bX.  
\label{ww.1}\end{equation}
Let $H_1:= \tbeta^{-1}(S_1)$ and $H_2:=\overline{\tbeta^{-1}(\pa\bX\setminus S_1)}$ be the corresponding boundary hypersurfaces.  Using the local coordinates  $(v_1,\ldots, v_n, \theta, \bx)$ near $S_1$ in $\bX$ given by
 $$
       v_i= \xi^{n-1}z_i,  \quad z_{n+1}= \frac{e^{i\theta}}{\xi^{n-k}},
 $$    
 we see that the equation defining  $\bW_{\epsilon,k}$ in $\bX$ is given by
 \begin{equation}
      P_k(v_1,\ldots,v_n)= \epsilon e^{i\theta}\xi^{n(k-1)}= \epsilon e^{i\theta} \bx^k.
 \label{para.1}\end{equation}
 On the other hand, $S_1$ is given by the equation $v_1=\ldots= v_n= \bx=0$, hence the blow-up of $S_1$ corresponds to introducing the coordinates $(\omega_1,\ldots, \omega_n, \omega_{n+1})\in \bbS^{2n}\subset \bbC^n\times \bbR$ and $r\in [0,\infty)$ such that 
 $$
         v_i= r\omega_i, \quad \bx= r\omega_{n+1}.
 $$
 In these coordinates, the closure $\tW_{\epsilon,k}$ of $W_{\epsilon,k}$ in $\tX$ is given by
 $$
   P_k(\omega_1,\ldots,\omega_n)= \epsilon e^{i\theta}\omega_{n+1}^k.
 $$
When we restrict this equation to the new boundary hypersurface $H_{1}$, that is, when we restrict to $r=0$, we obtain exactly the same equation.  In particular, for $\epsilon\ne 0$, the restriction of $\tW_{\epsilon,k}$ to $H_{1}$ is smooth.  
  
Alternatively, we can introduce adapted holomorphic coordinates near $S_1$.  Namely, fix $\theta_0\in \bbR$ and restrict to a region of the form 
\begin{equation}
 \mathcal{R}_{\theta_0}= \{ (z_1,\ldots,z_{n+1})\in \bbC^{n+1}\; | \; z_{n+1}\ne 0, \theta_0< \arg z_{n+1} < \theta_0+2\pi \},   
\label{w.11}\end{equation}
 so that $\log z_{n+1}$ and $z_{n+1}^{\frac{1}k}$ can be defined  holomorphically in that region.  On $\mathcal{R}_{\theta_0}$, one can then consider the holomorphic coordinates
 $\zeta_1,\ldots,\zeta_n, z_{n+1}$ defined by
 \begin{equation}
   \zeta_i:= \frac{z_i}{z_{n+1}^{\frac{1}k}}, \quad z_{n+1}.
 \label{w.12}\end{equation}
Setting $\zeta=(\zeta_1,\ldots,\zeta_n)$, notice that $\tx_2:= \langle \zeta \rangle^{-1}$ can be taken to be a boundary defining function of $H_{2}$ in $\tX$, where $\langle\zeta\rangle:=\sqrt{1+|\zeta|^2}$, so that if $\tx:= \tbeta^* \bx$, then $\tx_1:=\langle \zeta \rangle \tx$ is a boundary defining function of $H_{1}$.

In terms of these coordinates, $\tW_{\epsilon,k}$ is defined more simply by the equation
\begin{equation}
    P_k(\zeta_1,\ldots,\zeta_n)= \epsilon  
\label{w.13}\end{equation}
and does not degenerate when $|z_{n+1}|\to \infty$.  Along the branch locus $\arg z_{n+1}=\theta_0$, we can identify the hypersurface defined by \eqref{w.13} when $\arg z_{n+1}\searrow \theta_0$ with the one when $\arg z_{n+1}\nearrow \theta_0+2\pi$ via the automorphism of \eqref{w.13} given by
\begin{equation}
  \zeta_i \mapsto e^{\frac{2\pi \sqrt{-1}}{k}}\zeta_i.  
\label{w.13b}\end{equation}
For $z_{n+1}\ne 0$, we can then regard $W_{\epsilon,k}$ as a holomorphic family of hypersurfaces of degree $k$ over $\bbC\setminus \{0\}$ all isomorphic to \eqref{w.13}.  In particular, this gives $H_1\cap \tW_{\epsilon,k}$ the  natural structure of a fibre bundle
\begin{equation}
\xymatrix{  
   \overline{V}_{\epsilon,k} \ar[r] & H_1\cap\tW_{\epsilon,k} \ar[d]^{\phi_1}\\
            & S_1.
}
\label{w.13c}\end{equation}

On the other hand, the Kähler form of the metric $g_{W_{0,k}}$ is given by
\begin{equation}
  \omega_{W_{0,k}}= \frac{\sqrt{-1}}2 \pa \db u_{0,k} \quad \mbox{with} \; u_{0,k}= |z_{n+1}|^2 + \| z \|^{\frac{2(n-k)}{n-1}}_{N^*_{F_k}}, \; z=(z_1,\ldots,z_n).
\label{w.14}\end{equation}
Since the metric is singular when $z=0$, this metric is problematic near $H_1$.  
On the other hand, near $H_2$, but away from $H_1$, notice that it is in fact a conical metric with radial function given by 
$$
     R= \sqrt{|z_{n+1}|^2+ \|z \|_{N^*_{F_k}}^{\frac{2(n-k)}{n-1}}}\asymp \frac{1}{\xi^{n-k}},
$$
so in particular it behaves like an $\AC$-metric near $H_2$.  

Now, by the homogeneity of the norm $\| \cdot \|_{N^*_{F_k}}$, in the coordinates $(\zeta_1,\ldots, \zeta_n,z_{n+1})$, the potential $u_{0,k}$ is given by
\begin{equation}
 u_{0,k}= |z_{n+1}|^2 + |z_{n+1}|^{\frac{2(n-k)}{k(n-1)}} \| \zeta\|^{\frac{2(n-k)}{n-1}}_{N^*_{F_k}}, \quad \zeta:=(\zeta_1,\ldots, \zeta_n).
\label{w.15}\end{equation}
To construct a corresponding metric on $W_{\epsilon,k}$ for $\epsilon \ne 0$, this suggests to consider instead the potential 
\begin{equation}
   u_{\epsilon,k}:= |z_{n+1}|^2 + |z_{n+1}|^{\frac{2(n-k)}{k(n-1)}} \phi_{\epsilon,k}(\zeta),
\label{w.16}\end{equation}
where $\phi_{\epsilon,k}$ is the potential $v+w$ given in Theorem~\ref{ac.11} so that $\frac{\sqrt{-1}}2\pa\db\phi_{\epsilon,k}(\zeta)$ is the Kähler form of the Calabi-Yau $\AC$-metric $g_{V_{\epsilon,k}}$ on the affine hypersurface defined by \eqref{w.13}.  By the uniqueness of the potential $\phi_{\epsilon,k}$, notice that the automorphism \eqref{w.13b} is in fact an isometry of the Calabi-Yau $\AC$-metric $g_{V_{\epsilon,k}}$.  
\begin{remark}
Hence, the $\AC$-metric $g_{V_{\epsilon,k}}$ unambiguously defines a family of fibrewise $\AC$-metrics in the fibre bundle \eqref{w.13c}.
\label{w.16b}\end{remark}
Before studying the potential \eqref{w.16} in more detail, we need to introduce the class of metrics we wish to study.  We first need to recall the notion of $\Qb$-metrics introduced in \cite{CDR2016}.  To do this, we need to first change the smooth structure on $\tX$, so let $\cX$ be the manifold with corners which is $\tX$ as a topological space, but with algebra of smooth functions given by the smooth functions on the interior of $\tX$ admitting Taylor expansions at $H_1$ and $H_2$ in integer powers of $x_1:= \tx_1^{\frac{n-k}{n-1}}$ and $x_2:= \tx_2^{\frac{n-k}{n-1}}$.  We denote by $\cH_1$ and $\cH_2$ the boundary hypersurfaces of $\cX$ and by $x:= \tx^{\frac{n-k}{n-1}}=x_1 x_2$ the total boundary defining function.  Notice that $\cH_i$ is the same as $H_i$ as a topological space but has a different algebra of smooth functions.  We denote by $\cW_{\epsilon,k}$ the closure of $W_{\epsilon,k}$ in $\cX$.  Recall from \cite{Melrose1992} that 
 $$
      \cV_b(\cW_{\epsilon}):= \{   \xi \in \CI(\cW_{\epsilon,k};T\cW_{\epsilon,k})\; | \; \xi x_i\in x_i\CI(\cW_{\epsilon,k}) \quad \forall i\}
 $$
 is the \textbf{Lie algebra of $b$-vector fields} on $\cW_{\epsilon,k}$. 
 \begin{definition}[\cite{CDR2016}]
 The \textbf{Lie algebra of $\Qb$-vector fields} on $\cW_{\epsilon,k}$ is given by the vector fields $\xi\in \cV_b(\cW_{\epsilon,k})$ such that 
 \begin{itemize} 
 \item $\left. \xi\right|_{\cH_1\cap \cW_{\epsilon,k}}$ is tangent to the fibres of $\phi_1: \cH_1\cap\cW_{\epsilon,k}\to S_1$ induced by \eqref{w.13c};
 \item  $\xi (x)\in x_1^2 x_2 \CI(\cW_{\epsilon,k})$.  
 \end{itemize} 
 \label{qb.1}\end{definition}
 \begin{remark}
 This definition depends in general on the choice of total boundary defining function $x$.  As shown in \cite[Lemma~1.16]{CDR2016}, another choice of total boundary defining function $x'$ will induce the same Lie algebra if and only if $\frac{x}{x'}$ is constant along the fibres of $\phi_1: \cH_1\cap \cW_{\epsilon,k}\to S_1$.  
 \label{qb.1b}\end{remark}
Near $\cH_1\cap \cH_2$, using the coordinates $x_1,x_2,\theta, y_1,\ldots, y_{2n-3}$ with $\theta$ the coordinate on $S_1$ and $y=(y_1,\ldots, y_{2n-3})$ coordinates on $\pa\bV_{\epsilon,k}$, we see that the Lie algebra $\cV_{\Qb}(\cW_{\epsilon,k})$ is locally spanned over $\CI(\cW_{\epsilon,k})$ by
\begin{equation}
      x_1^2\frac{\pa}{\pa x_1}, x_2\frac{\pa}{\pa x_2}- x_1\frac{\pa}{\pa x_1}, x_1\frac{\pa}{\pa \theta}, \frac{\pa}{\pa y_1},\ldots, \frac{\pa}{\pa y_{2n-3}}.
\label{qb.2}\end{equation}
Near $\cH_2$ but away from $\cH_1$, it is the same as the Lie algebra of $b$-vector fields, whereas near $\cH_1$ but away from $\cH_2$, it is the same as the Lie algebra of $\phi_1$-vector fields of Mazzeo-Melrose \cite{Mazzeo-MelrosePhi}.  Thus, globally on $\cW_{\epsilon,k}$, $\cV_{\Qb}(\cW_{\epsilon,k})$ corresponds to the sections of a locally free $\CI(\cX)$-module of rank $2n-2$.  This means that there exists a natural smooth vector bundle ${}^{\Qb}T \cW_{\epsilon,k}\to \cW_{\epsilon,k}$, called the \textbf{$\Qb$-tangent bundle}, such that there is a canonical identification
\begin{equation}
     \cV_{\Qb}(\cW_{\epsilon,k})= \CI(\cW_{\epsilon,k};{}^{\Qb}T\cW_{\epsilon,k}).
\label{qb.3}\end{equation}  
In fact,  the bundle ${}^{\Qb}T\cW_{\epsilon,k}$ is naturally a Lie algebroid with anchor map $\iota_{\Qb}: {}^{\Qb}T\cW_{\epsilon,k}\to T\cW_{\epsilon,k}$ inducing, via the identification \eqref{qb.3},  the natural inclusion 
$$
\cV_{\Qb}(\cW_{\epsilon,k})\subset \CI(\cW_{\epsilon,k}; T\cW_{\epsilon,k}).
$$
We denote by ${}^{\Qb}T^*\cW_{\epsilon,k}$ the dual of ${}^{\Qb}T\cW_{\epsilon,k}$ and call it the \textbf{$\Qb$-cotangent bundle} of $\cW_{\epsilon,k}$.  Near $\cH_1\cap\cH_2$, notice that ${}^{\Qb}T^*\cW_{\epsilon,k}$ is locally spanned by
\begin{equation}
 \frac{dx}{x_1^2x_2}, \frac{d\theta}{x_1}, \frac{dx_2}{x_2}, dy_1, \ldots, dy_{2n-3}.
\label{qb.4}\end{equation}
\begin{definition}[\cite{CDR2016}]
A $\Qb$-metric is a choice of Euclidean metric $g_{\Qb}$ for the vector bundle ${}^{\Qb}T\cW_{\epsilon,k}$.  A \textbf{smooth $\Qb$-metric} on $W_{\epsilon,k}$ is a Riemannian metric on $W_{\epsilon,k}$  induced by some $\Qb$-metric via the map $\iota_{\Qb}: {}^{\Qb}T \cW_{\epsilon,k}\to T \cW_{\epsilon,k}$.  We say instead that a Riemannian metric $g$ on $W_{\epsilon,k}$  is a \textbf{polyhomogeneous $\Qb$-metric} if it is induced by a Euclidean metric $g_{\Qb}$ of ${}^{\Qb}T\cW_{\epsilon,k}$ which is polyhomogeneous as a section of $S^2({}^{\Qb}T^*\cW_{\epsilon,k})$.  Finally, a $\Qb$-metric on $W_{\epsilon,k}$ is a Riemannian metric $g$ on $W_{\epsilon,k}$ such that
$g$ is quasi-isometric to some smooth $\Qb$-metric $g_{\Qb}$ and if all its covariant derivatives are bounded with respect to $g_{\Qb}$.      
\label{qb.5}\end{definition}

\begin{lemma}[Proposition~1.30 in \cite{CDR2016}]
Any $\Qb$-metric on $W_{\epsilon,k}$ is a complete metric of infinite volume with bounded geometry. 
\label{qb.6}\end{lemma}
 
We can now define the class of metrics that will be the subject of interest in this paper.    

\begin{definition}
On $W_{\epsilon,k}$, a \textbf{warped $\QAC$-metric} is a Riemannian metric $g_w$ of the form 
\begin{equation}
         g_w = \frac{g_{\Qb}}{\chi^2 }
\label{conf.1}\end{equation}
for some $\Qb$-metric $g_{\Qb}$, where
 \begin{equation}
 \chi:= \frac{x^{\frac{k(n-1)}{n(k-1)}}}{x_1}= x^{\frac{n-k}{n(k-1)}}x_2 =x_1^{\frac{n-k}{n(k-1)}}x_2^{\frac{k(n-1)}{n(k-1)}}.
 \label{gb.13b}\end{equation}
We say that $g_w$ is \textbf{smooth} or \textbf{polyhomogeneous} if $g_{\Qb}$ is smooth or polyhomogeneous.  
\label{qb.11}\end{definition}

\begin{proposition}
  Any warped $\QAC$-metric on $W_{\epsilon,k}$ is complete of infinite volume with bounded geometry.
\label{gb.12}\end{proposition}
\begin{proof}
Since $\chi$ is a bounded positive function, we see from Lemma~\ref{qb.6} that any warped $\QAC$-metric is complete of infinite volume with positive injectivity radius.  From the local basis \eqref{qb.2}, we see also that in terms of any $\Qb$-metric, the covariant derivatives of $\chi$ are bounded.  Hence, by Lemma~\ref{qb.6}, the curvature of any warped $\QAC$-metric is bounded, as well as all its covariant derivatives.  
\end{proof}

As for $\Qb$-metrics, we can naturally associate to a warped $\QAC$-metric a Lie algebra of warped $\QAC$-vector fields, namely
$$
    \cV_w(\cW_{\epsilon,k}):= \chi\cV_{\Qb}(\cW_{\epsilon,k}).
$$
Since $\chi$ is not smooth up to the boundary of $\cW_{\epsilon,k}$, it is not a subalgebra of $\CI(\cW_{\epsilon,k};T\cW_{\epsilon,k})$, but it is nevertheless a $\CI(\cW_{\epsilon,k})$-module.  In fact, there is a natural vector bundle over $\cW_{\epsilon,k}$ that we call the \textbf{warped $\QAC$-tangent bundle} and denote by ${}^{w}T\cW_{\epsilon,k}$, such that there is  a canonical identification
$$
       \CI(\cW_{\epsilon,k};{}^{w}T\cW_{\epsilon,k})= \cV_w(\cW_{\epsilon,k}).
$$   
Its dual ${}^{w}T^*\cW_{\epsilon,k}$ is the \textbf{warped $\QAC$-cotangent bundle}.  From \eqref{qb.4}, we see that near $\cH_1\cap \cH_2$, a basis of local sections of ${}^{w}T^*\cW_{\epsilon,k}$ is given by
\begin{equation}
 \frac{dx}{x_1^2x_2\chi}= \frac{dx}{x^{\frac{k(n-1)}{n(k-1)}+1}}, \frac{d\theta}{x_1\chi}= \frac{d\theta}{x^{\frac{k(n-1)}{n(k-1)}}}, \frac{dx_2}{x_2\chi}= \frac{dx_2}{x_2^2 x^{\frac{n-k}{n(k-1)}}}, \frac{dy_1}{\chi}=\frac{dy_1}{x_2x^{\frac{n-k}{n(k-1)}}}, \ldots, \frac{dy_{2n-3}}{\chi}=\frac{dy_{2n-3}}{x_2x^{\frac{n-k}{n(k-1)}}}.
 \label{gb.13}\end{equation}
Hence, in these coordinates, we see that an example of a warped $\QAC$-metric is given by
\begin{equation}
g_{w}= \frac{dx^2}{x^{2\frac{k(n-1)}{n(k-1)}+2}}+ \frac{d\theta^2}{x^{2\frac{k(n-1)}{n(k-1)}}}+ \frac{dx_2^2}{x_2^4 x^{2\frac{n-k}{n(k-1)}}} + \sum_{i=1}^{2n-3} \frac{dy_i^2}{x_2^2 x^{2\frac{n-k}{n(k-1)}}}.
\label{qb.13c}\end{equation}
Making the change of variables $x= r^{-\frac{n(k-1)}{k(n-1)}}$, notice that this can be written more simply as
\begin{equation}
g_w= c^2dr^2 + r^2 d\theta^2 + r^{\frac{2(n-k)}{k(n-1)}}\left( \frac{dx_2^2}{x_2^4} + \sum_{i=1}^{2n-3} \frac{dy_i^2}{x_2^2} \right) \quad \mbox{with} \quad c:=  \frac{n(k-1)}{k(n-1)}.
\label{qb.13d}\end{equation}
In this form, $g_w$ is a warped product of $\AC$-metrics.  This differs from the model for $\QAC$-metrics, which in this lower depth setting would be a Cartesian product of $\AC$-metrics, namely 
\begin{equation}
  g_{\QAC}= dr^2 + r^2d\theta^2 + \left( \frac{dx_2^2}{x_2^4} + \sum_{i=1}^{2n-3} \frac{dy_i^2}{x_2^2} \right);
 \label{qb.13e}\end{equation}  
 see \cite[\S~2.3.5]{DM2014} and \cite[Example~1.22 and (1.9)]{CDR2016}.

There are various functional spaces one can associated to a $\Qb$-metric or a warped $\QAC$-metric.  To describe them, recall that given a complete Riemannian manifold $(M,g)$ and a Euclidean vector bundle $E\to M$ with a connection $\nabla$ compatible with the Euclidean structure, one can for each $\ell\in \bbN_0$ associate the space $\cC^\ell_g(M;E)$ comprising continuous 
sections $\sigma: M\to E$ such that
\begin{equation}
  \nabla^j\sigma\in \cC^0(M; T^0_jM \otimes E), \quad \mbox{and} \quad  \sup_{p\in M} |\nabla^j\sigma|_g<\infty,  \quad \forall j\in \{0,\ldots,\ell\},
  \label{qb.7}\end{equation}  
where $\nabla$ denotes the covariant derivative induced by the Levi-Civita connection of $g$ and the connection on $E$, $|\cdot|_g$ is the norm induced by $g$ and the Euclidean structure on $E$ and $T^0_jM = T^*M^{\otimes j}$.  This is in fact a Banach space with norm 
\begin{equation}
     \| \sigma \|_{g,\ell} := \sum_{j=0}^{\ell} \sup_{p\in M} |\nabla^j \sigma(p)|_g.
\label{qb.8}\end{equation}
Taking the intersection over all $\ell$ yields the Fréchet space
\begin{equation}
  \CI_g(M;E):= \bigcap_{\ell\in \bbN_0} \cC^\ell_g(M;E).
\label{qb.9}\end{equation}
For $\alpha\in (0,1]$ and $\ell\in \bbN_0$, we can also consider the Hölder space $\cC^{\ell,\alpha}_g(M;E)$ of sections $\sigma\in \cC^\ell_g(M;E)$ such that 
$$
       [\nabla^\ell\sigma]_{g,\alpha}:= \sup \left\{ \frac{|P_{\gamma}(\nabla^\ell\sigma(\gamma(0)))-\nabla^\ell\sigma(\gamma(1))|}{\ell(\gamma)^{\alpha}} \quad | \quad \gamma\in \CI([0,1];M), \; \gamma(0)\ne \gamma(1) \right\}<\infty,
$$
where $P_{\gamma}: \left.  T^0_{\ell} M\otimes E\right|_{\gamma(0)}\to \left.  T^0_\ell M\otimes E\right|_{\gamma(1)}$ is the parallel transport along $\gamma$ and $\ell(\gamma)$ is the length of $\gamma$ with respect to $g$.  This is also a Banach space with norm given by
\begin{equation}
    \| \sigma\|_{g,\ell,\alpha}:= \| \sigma \|_{g,\ell}+ [\nabla^\ell\sigma]_{g,\alpha}.
\label{qb.10}\end{equation}
For $\mu\in \CI(M)$ a positive function, we can also consider the weighted version
$$
    \mu\cC^{\ell,\alpha}_g(M;E):= \left\{ \sigma \; | \;  \frac{\sigma}\mu \in \cC^{\ell,\alpha}_g(M;E) \right\} \quad \mbox{with norm} \quad  \| \sigma \|_{\mu\cC^{\ell,\alpha}_g}:= \left\| \frac{\sigma}\mu \right\|_{g,\ell,\alpha}.
$$
Similarly, we can consider the space $L^2_{g}(M;E)$ of measurable sections $\sigma$ such that
$$
          \| \sigma\|^2_{L^2_{g}}:= \int_{M} |\sigma|_g^2 dg <\infty,
$$
where $dg$ is the volume density of $g$, and for $\ell\in \bbN$ the Sobolev space $H^\ell_{g}(M;E)$ of sections $\sigma\in L^2_{g}(M;E)$ such that $\nabla^j\sigma \in L^2_{g}(M; T^0_j M \otimes E)$ for $j\in \{0,\ldots, \ell\}$.  

By choosing $M=W_{\epsilon,k}$ and $g=g_{\Qb}$ to be a $\Qb$-metric, we obtain in this way the $\Qb$-Hölder space $\cC_{\Qb}^{\ell,\alpha}(W_{\epsilon,k};E)$ as well as the space $\cC^\ell_{\Qb}(W_{\epsilon,k})$.   Similarly, by choosing $g=g_w$ to be a warped $\QAC$-metric, we can define the warped $\QAC$-Hölder space $\cC_{w}^{\ell,\alpha}(W_{\epsilon,k};E)$.  There is an obvious continuous inclusion $\cC_{\Qb}^{\ell,\alpha}(W_{\epsilon,k})\subset \cC^{\ell,\alpha}_{w}(W_{\epsilon,k};E)$.  Conversely, \cf \cite[Lemma~{1.31}]{CDR2016}, the following partial counterpart will be useful to solve the complex Monge-Ampère equation in \S~\ref{cy.0}.

\begin{lemma}
For $0<\delta<1$, there is a continuous inclusion $\chi^{\delta}\cC^{0,1}_{w}(W_{\epsilon,k};E)\subset \cC^{0,\alpha}_{\Qb}(W_{\epsilon,k};E)$ for $\alpha\le \delta$.
\label{inc.1}\end{lemma}
\begin{proof}
Let $g_{\Qb}$ be a choice of $\Qb$-metric and let $g_w:= \frac{g_{\Qb}}{\chi^2}$ be the corresponding warped $\QAC$-metric.  Notice first from \eqref{qb.2} and \eqref{qb.4} that 
\begin{equation}
d\chi\in \chi\cA_{\phg}(\cW_{\epsilon,k};{}^{\Qb}T^*\cW_{\epsilon})\subset \chi \cC^{\infty}_{\Qb}(W_{\epsilon,k};{}^{\Qb}T\cW_{\epsilon,k}).
\label{inc.2}\end{equation}
In particular, we have that  
\begin{equation}
       d\log \chi= \frac{d\chi}{\chi}\in \CI_{\Qb}(W_{\epsilon,k};{}^{\Qb}T^*\cW_{\epsilon,k})
\label{inc.3}\end{equation}
and that 
\begin{equation}
     d\chi^{-\delta}= \frac{-\delta d\chi}{\chi^{1+\delta}}\in \chi^{-\delta}\CI_{\Qb}(W_{\epsilon,k};{}^{\Qb}T^*\cW_{\epsilon,k}).
\label{inc.4}\end{equation}
Now, given any path $\gamma\in \CI([0,1];W_{\epsilon,k})$ with $\gamma(0)\ne \gamma(1)$, let $\ell_{\Qb}(\gamma)$ and $\ell_{w}(\gamma)$ be the length of $\gamma$ with respect to the metrics $g_{\Qb}$ and $g_{w}$.  Denoting by $d\gamma_{\Qb}$ the measure on $[0,1]$ induced by $\gamma$ and the metric $g_{\Qb}$, we see using \eqref{inc.4} and the Hölder inequality with $p=\frac{1}\delta$ and $q= \frac{1}{1-\delta}$ that
\begin{equation}
\begin{aligned}
  \left| \frac{1}{\chi(\gamma(0))^{\delta}}- \frac{1}{\chi(\gamma(1))^{\delta}}\right| & \le \delta\int_0^1 \frac{| d\chi(\gamma(t))|_{g_{\Qb}}}{\chi(\gamma(t))^{1+\delta}}d\gamma_{\Qb} \le K\delta\int \frac{d\gamma_{\Qb}}{\chi(\gamma(t))^{\delta}} \\
  & \le K\delta \left( \int_{0}^1\frac{d\gamma_{\Qb}}{\chi(\gamma(t))} \right)^{\delta}
  \left( \int_0^1 d\gamma_{\Qb}  \right)^{1-\delta}= K\delta \ell_w(\gamma)^{\delta}\ell_{\Qb}(\gamma)^{1-\delta},
\end{aligned}  
\label{inc.5}\end{equation}  
where $K:= \|d\log \chi \|_{g_{\Qb},0}$.  

On the other hand, given $\sigma\in \chi^{\delta}\cC^{0,1}_{w}(W_{\epsilon,k};E)$, consider the positive constant
$C:= \left\| \chi^{-\delta}\sigma \right\|_{g_w,0,1}$.  Then we have that
\begin{equation}
\begin{aligned}
2C \min\left\{ \ell_w(\gamma),1\right\} & \ge \left| \frac{P_{\gamma}(\sigma(\gamma(0)))}{\chi(\gamma(0))^{\delta}}- \frac{\sigma(\gamma(1))}{\chi(\gamma(1))^{\delta}} \right| \\
&=  \left| \frac{P_{\gamma}(\sigma(\gamma(0)))}{\chi(\gamma(0))^{\delta}}-\frac{\sigma(\gamma(1))}{\chi(\gamma(0))^{\delta}}+ \frac{\sigma(\gamma(1))}{\chi(\gamma(0))^{\delta}} -\frac{\sigma(\gamma(1))}{\chi(\gamma(1))^{\delta}}  \right| \\
& \ge \frac{|P_{\gamma}(\sigma(0))-\sigma(\gamma(1))|}{\chi(\gamma(0))^{\delta}}- \|\sigma\|_{g_w,0}\left| \chi(\gamma(0))^{-\delta}- \chi(\gamma(1))^{-\delta} \right| \\
& \ge \frac{|P_{\gamma}(\sigma(0))-\sigma(\gamma(1))|}{\chi(\gamma(0))^{\delta}} -CK\delta \ell_w(\gamma)^{\delta}\ell_{\Qb}(\gamma)^{1-\delta},
\end{aligned}
\label{inc.6}\end{equation}
where we have used \eqref{inc.5} in the last step, as well as the inequality $\|\sigma\|_{g_w,0}\le C$ assuming without loss of generality that $\chi\le 1$.  In other words, we see that \eqref{inc.6} gives the estimate
\begin{equation}
|P_{\gamma}(\sigma(0))-\sigma(\gamma(1))|\le 2C\chi(\gamma(0))^{\delta} \min\left\{ \ell_w(\gamma),1\right\}+ CK\delta\chi(\gamma(0))^{\delta} \ell_w(\gamma)^{\delta}\ell_{\Qb}(\gamma)^{1-\delta}.
\label{inc.7}\end{equation}  
Hence, if $\ell_{\Qb}(\gamma)<1$, let $t_{\min}\in [0,1]$ be a point where $\chi\circ\gamma\in \CI([0,1])$ attains its minimum, so that by \eqref{inc.3},
$$
\left| \log\left(\frac{\chi(\gamma(t_{min}))}{\chi(\gamma(0))}\right) \right|\le K \ell_{\Qb}(\gamma)\le K  \quad \Longrightarrow \quad \frac{\chi(\gamma(0))}{\chi(\gamma(t_{\min}))} \le e^K,
$$
where $K$ is as in \eqref{inc.5}.  In this case, since $\alpha\le \delta$, we see that \eqref{inc.7} gives
\begin{equation}
\begin{aligned}
\frac{|P_{\gamma}(\sigma(0))-\sigma(\gamma(1))|}{\ell_{\Qb}(\gamma)^{\alpha}} & \le
\frac{2C\chi(\gamma(0))^{\delta} \min\left\{ \ell_w(\gamma),1\right\}}{\ell_{\Qb}(\gamma)^{\alpha}} + \frac{CK\delta\chi(\gamma(0))^{\delta} \ell_w(\gamma)^{\delta}\ell_{\Qb}(\gamma)^{1-\delta}}{\ell_{\Qb}(\gamma)^{\alpha}} \\
 & \le \frac{2C\chi(\gamma(0))^{\delta} \min\left\{ \ell_w(\gamma),1\right\}}{\chi(\gamma(t_{\min}))^{\alpha}\ell_{w}(\gamma)^{\alpha}}+ \frac{CK\delta\chi(\gamma(0))^{\delta} \ell_w(\gamma)^{\delta}\ell_{\Qb}(\gamma)^{1-\alpha}}{\chi(\gamma(t_{\min}))^{\delta}\ell_{w}(\gamma)^{\delta}}  \\
 &\le 2Ce^{K\alpha} + CK\delta e^{K\delta}\ell_{\Qb}(\gamma)^{1-\alpha}\le (2+K\delta)C e^{K\delta}.
\end{aligned}
\label{inc.8}\end{equation}  
If instead $\ell_{\Qb}(\gamma)\ge 1$, then we have more simply that
\begin{equation}
\frac{|P_{\gamma}(\sigma(0))-\sigma(\gamma(1))|}{\ell_{\Qb}(\gamma)^{\alpha}}\le |P_{\gamma}(\sigma(0))-\sigma(\gamma(1))|\le 2 \| \sigma\|_{g_w,0} \le 2C\le 2C+ CK\delta\le (2+ K\delta)Ce^{K\delta}.
\label{inc.9}\end{equation}
Hence, combining \eqref{inc.8} with \eqref{inc.9} and taking the supremum over $\gamma$ yields
$$
     [\sigma]_{g_{\Qb},0,\alpha}\le (2+K\delta)Ce^{K\delta}= (2+K\delta)e^{K\delta} \left\| \chi^{-\delta}\sigma \right\|_{g_w,0,1},
$$
from which the result follows.  

\end{proof}

We can also consider the Sobolev space $H_{\Qb}^\ell(W_{\epsilon,k})$ associated to a $\Qb$-metric.  For a warped $\QAC$-metric, instead of the natural Sobolev space associated to such a metric, we will consider the weighted version of the $\Qb$-Sobolev space
$$
         H^\ell_w(W_{\epsilon,k};E):= \chi^n H^\ell_{\Qb}(W_{\epsilon,k};E).
$$
The factor $\chi^n$ ensures that we integrate with respect to the volume density of a warped $\QAC$-metric, but since we use a weighted $\Qb$-Sobolev space, the pointwise norms of the derivatives are measured with respect to a $\Qb$-metric instead of a warped $\QAC$-metric.

As for $\QAC$-metrics, we can consider weighted versions of this Sobolev space.  First, consider the strictly positive functions  $\rho,w$ on $\bbC^{n+1}$ defined by
\begin{equation}
    \rho:= x^{-\frac{k(n-1)}{n(k-1)}}=\tbeta^{*}\left( \frac{1}{\xi^{n-k}} \right)  \quad \mbox{outside a compact set} 
\label{ww.2}\end{equation}
and 
\begin{equation}
 w:=x_1= \frac{ \left(|z_{n+1}|^{\frac{1}{k}} \langle \zeta \rangle\right)^{\frac{n-k}{n-1}} }{|z_{n+1}|}=  \frac{\langle \zeta \rangle^{\frac{n-k}{n-1}}}{|z_{n+1}|^{\frac{n(k-1)}{k(n-1)}}}  \quad \mbox{near $\cH_1$,}
\label{ww.3}\end{equation}
where without loss of generality we can take $x=|z_{n+1}|^{-\frac{n(k-1)}{k(n-1)}}$ near $\cH_1$.
Notice in particular that $\rho w= \chi^{-1}$, so that \eqref{conf.1} can be written as
$$
         g_w= (\rho w)^2 g_{\Qb},
$$
in direct analogy with the corresponding relation for $\QAC$-metrics.  
In terms of these weight functions, we will be interested in the weighted Sobolev spaces  
\begin{equation}
           \rho^{\delta+n} w^{\tau + n-1}H^{\ell}_w(W_{\epsilon,k} )
\label{ww.4}\end{equation}
and the weighted Hölder spaces 
\begin{equation}
\rho^{\delta}w^{\tau} \cC^{\ell,\alpha}_{\Qb}(W_{k,\epsilon})
\end{equation}
for $\ell\in \bbN_0$ and $\delta,\tau\in \bbR$.

We now return to the potential \eqref{w.16}.   
 \begin{lemma}
For $\epsilon\ne 0$, there exists an open neighborhood $\cU_1$ of $\cH_1\subset \cX$ such that $\frac{\sqrt{-1}}{2} \pa\db u_{\epsilon,k}$ is the Kähler form on $\cU_1\cap W_{\epsilon,k}$ of a warped $\QAC$-metric.  
\label{w.17}\end{lemma}
\begin{proof}
The form is clearly closed, so to see that it is Kähler, it suffices to check that it is positive definite near $\cH_1$.  We compute that 
$$
   \frac{\sqrt{-1}}2 \pa \db u_{\epsilon,k} = \frac{\sqrt{-1}}2 dz_{n+1}\wedge d\overline{z}_{n+1} + |z_{n+1}|^{\frac{2(n-k)}{k(n-1)}} \frac{\sqrt{-1}}{2}\pa\db \phi_{\epsilon,k}(\zeta) + Q
$$
where
$$
 Q  :=\phi_{\epsilon,k}(\zeta) \frac{\sqrt{-1}}2 \pa \db |z_{n+1}|^{\frac{2(n-k)}{k(n-1)}}+ \frac{\sqrt{-1}}2 \pa \phi_{\epsilon,k}(\zeta) \wedge\db |z_{n+1}|^{\frac{2(n-k)}{k(n-1)}}+ 
  \frac{\sqrt{-1}}2\pa |z_{n+1}|^{\frac{2(n-k)}{k(n-1)}}  \wedge \db\phi_{\epsilon,k}(\zeta).
$$
Clearly, since $\frac{\sqrt{-1}}{2} \pa \db \phi_{\epsilon,k}(\zeta)>0$ on $V_{\epsilon,k}$, the $2$-form
$$
\frac{\sqrt{-1}}2 dz_{n+1}\wedge d\overline{z}_{n+1} + |z_{n+1}|^{\frac{2(n-k)}{k(n-1)}} \frac{\sqrt{-1}}2\pa\db \phi_{\epsilon,k}(\zeta)
$$
is positive definite.  Let $h$ be the corresponding Hermitian metric.  Using the polar coordinates $z_{n+1}=re^{i\theta}$, notice that $h$  is a warped product of $\AC$-metrics,
\begin{equation}
  h= dr^2 + r^2d\theta^2 + r^{\frac{2(n-k)}{k(n-1)}} g_{V_{\epsilon,k}},
\label{wa.1}\end{equation}
where the metric $g_{V_{\epsilon,k}}$ is seen as a family of fibrewise $\AC$-metrics in the fibre bundle \eqref{w.13c} as described in Remark~\ref{w.16b}.  In particular, comparing with \eqref{qb.13d}, we see that $h$ is a warped $\QAC$-metric.  
Using $h$ to compute the norm of $Q$ and taking into account the fact that 
$$
        \phi_{\epsilon,k}(\zeta) \asymp \| \zeta \|^{\frac{2(n-k)}{n-1}}_{N^*_{F_k}}\asymp \langle\zeta\rangle^{\frac{2(n-k)}{n-1}}= x_2^{-2} \quad \mbox{as} \; |\zeta|\to \infty,
$$
we see that 
\begin{equation}
  Q\in |z_{n+1}|^{\frac{n-k}{k(n-1)}-1 } \langle\zeta\rangle^{\frac{(n-k)}{n-1}}\cA_{\phg}(\cW_{\epsilon,k};\Lambda^2({}^{w}T^*\cW_{\epsilon,k}\otimes \bbC))= x_1\cA_{\phg}(\cW_{\epsilon,k};\Lambda^2({}^{w}T^*\cW_{\epsilon,k}\otimes \bbC)),
  \label{w.18}\end{equation}
from which the result follows.
\end{proof}

\section{Construction of an asymptotically  Calabi-Yau metric on $W_{\epsilon,k}$ } \label{acy.0}

To determine if $\frac{\sqrt{-1}}{2}\pa \db u_{\epsilon,k}$ is asymptotically Calabi-Yau as we approach $\cH_1$,  let us look at the natural holomorphic volume form on $W_{\epsilon,k}$.  It is defined implicitly by 
$$
      \left. dz_1\wedge\ldots \wedge dz_n\wedge dz_{n+1} \right|_{W_{\epsilon,k}}= \left. \Omega^n_{\epsilon}\wedge d\left( P_k(z_1,\ldots,z_n) - \epsilon z_{n+1} \right)\right|_{W_{\epsilon,k}}.
$$
For $\epsilon=0$, notice in particular that $\Omega^n_0= -\Omega^{n-1}_0\wedge dz_{n+1}.$  Now, in terms of the coordinates $(\zeta_1,\ldots, \zeta_n, z_{n+1})$, notice that 
$$
      dz_1\wedge \ldots \wedge dz_n\wedge dz_{n+1}= z_{n+1}^{\frac{n}k}d\zeta_1\wedge \ldots \wedge d\zeta_n\wedge dz_{n+1}.
$$
Let us denote by $\Omega^{n-1}_{\epsilon}(\zeta)$ the holomorphic volume form defined implicitly on \eqref{w.13} by 
$$
    \left. d\zeta_1\wedge \ldots \wedge d\zeta_n \right|_{V_{\epsilon,k}}
 = \left.  \Omega^{n-1}_{\epsilon}(\zeta)\wedge d\left( P_k(\zeta_1,\ldots,\zeta_n)-\epsilon\right)\right|_{V_{\epsilon,k}}.
$$
Then, from the fact that 
$$
\left.  d\left( P_k(z_1,\ldots,z_n)  -\epsilon z_{n+1} \right)\right|_{W_{\epsilon,k}}= \left.  z_{n+1} d\left( P_k(\zeta_1,\ldots,\zeta_n)-\epsilon \right)\right|_{W_{\epsilon,k}},
$$
we see that
\begin{equation}
   \Omega^n_{\epsilon}= -z_{n+1}^{\frac{n}{k}-1} \Omega^{n-1}_{\epsilon}( \zeta)\wedge dz_{n+1}.
\label{w.19}\end{equation}
In particular, we see that
\begin{equation}
 \Omega^n_{\epsilon}\wedge \overline{\Omega}^n_{\epsilon}= |z_{n+1}|^{\frac{2(n-k)}{k}}\Omega^{n-1}_{\epsilon}(\zeta)\wedge dz_{n+1}\wedge \overline{\Omega}^{n-1}_{\epsilon}(\overline{\zeta})\wedge d\overline{z}_{n+1}.
\label{w.20}\end{equation}
In the case $\epsilon=0$, we know that $g_{W_{0,k}}$ is Calabi-Yau and that there is a constant $c_n$ such that 
$$
    \omega^n_{W_{0,k}}= c_n \Omega^n_0 \wedge \overline{\Omega}^n_0.
$$
Hence, we see from \eqref{w.18}, \eqref{ac.9b} and \eqref{w.20} that the following lemma holds.

\begin{lemma}
  The Kähler metric $\frac{\sqrt{-1}}2 \pa \db u_{\epsilon,k}$ is asymptotically Calabi-Yau near $\cH_1$ in the sense that 
  $$
        \log \left( \frac{(\frac{\sqrt{-1}}2 \pa \db u_{\epsilon,k})^n}{c_n \Omega^n_{\epsilon}\wedge \Omega_{\epsilon}^n} \right)\in x_1\cA_{\phg}(\cW_{\epsilon,k})\subset x_1\CI_{\Qb}(W_{\epsilon,k}).
  $$
\label{w.21}\end{lemma}

Near $\cH_{2}$ in  $\cX$, we shall check now that $\frac{\sqrt{-1}}2 \pa \db u_{\epsilon,k}$ is behaving asymptotically like $\omega_{W_{0,k}}$.  First, recall from \S~\ref{ac.0} the diffeomorphism 
$\varphi_{\epsilon}: \overline{\cU}_{0}\to \overline{\cU}_{\epsilon}$ such that \eqref{ac.8} and \eqref{ac.9} hold, where $\overline{\cU}_{\epsilon}= \overline{\cU}\cap \bV_{\epsilon,k}$ and $\overline{\cU}\subset \overline{\bbC^n}$ is an open neighborhood of $\beta^{-1}(F_k)\subset \pa \overline{\bbC^n}$.  Consider now a small open neighborhood $\cU^{12}$ of $H_{1}\cap H_{2}$ in  $\tX$ and define $\cU^{12}_{\epsilon}$ to be the restriction of $\cU^{12}$ to $\tW_{\epsilon,k}$ in $\tX$.   Then, in the coordinates $(\zeta,z_{n+1})$, we see that the diffeomorphism $\varphi_{\epsilon}$ induces a diffeomorphism $\varphi^{12}_{\epsilon}: \cU^{12}_0\to \cU^{12}_{\epsilon}$ such that 
\begin{equation}
   (\varphi^{12}_{\epsilon})^* J_{\epsilon}- J_0\in |\zeta|^{-k} \cA_{\phg}(\cU^{12}_0; \End({}^{w}T\cU^{12}_0)),
\label{w.22}\end{equation}
and  
\begin{equation}
 \log\left( \frac{(\varphi^{12}_{\epsilon})^* \Omega^n_{\epsilon}}{ \Omega^n_0}\right) \in |\zeta|^{-k}\cA_{\phg}(\cU^{12}_0), 
 \label{w.23}\end{equation}
 where $J_{\epsilon}$ denotes the complex structure of $W_{\epsilon,k}$. 
Moreover, from \eqref{ac.9b} and \eqref{ac.12}, we see that 
  \begin{equation}
  (\varphi^{12}_{\epsilon})^* \omega_{W_{\epsilon,k}}- \omega_{W_{0,k}}  \in |\zeta|^{-\mu\frac{n-k}{n-1}}\cA_{\phg}(\cU^{12}_0; \Lambda^2({}^{w}T^*\cU^{12}_0)).
  \label{w.24}\end{equation}
  \begin{remark}
  The map $\varphi^{12}_{\epsilon}: \cU^{12}_0\to \cU^{12}_{\epsilon}$ is a diffeomorphism for the smooth structure induced by $\tX$.  For the one induced by $\cX$, it is a diffeomorphism away from $\cH_2$, but is only a polyhomogeneous map if we include $\cU^{12}\cap \cH_2$.  On the other hand, notice that $\cA_{\phg}(\tW_{\epsilon,k};E)=\cA_{\phg}(\cW_{\epsilon,k};E)$, so in \eqref{w.22}, \eqref{w.23} and \eqref{w.24}, we can interchangeably  use the smooth structures induced by $\tX$ or $\cX$.   
  \label{w.24b}\end{remark}
  
  On the other hand,  let $\cU_2$ be an open set such that $\cU_2\cap \cH_1= \emptyset$  and $\cH_2 \subset \cU^{12}\cup \cU_2$.  Then set $\cU_{2,\epsilon}:= \cU_2\cap \tW_{\epsilon,k}$.  Using the fact that on $\beta\circ \tbeta(\cU_2)$, $\bbP W_{\epsilon,k}$ and $\bbP W_{0,k}$ are tangent to order $n(k-1)$ on their intersection on $\bbC\bbP^n_a\subset \bbC\bbP^{n+1}_{(1,a)}$, we see, \cf \cite{ChiLi}, that there exists a smooth diffeomorphism $\varphi_{2,\epsilon}: \cU_{2,0}\to \cU_{2,\epsilon}$ for the smooth structure induced by $\overline{\bbC^{n+1}_a}$ such that
\begin{equation}
  \begin{gathered}
  (\varphi^{2}_{\epsilon})^* J_{\epsilon}- J_0\in \xi^{n(k-1)} \cA_{\phg}(\cU_{2,0}; \End({}^{w}T\cU_{2,0})),  \\
  \log\left(\frac{(\varphi^{2}_{\epsilon})^* \Omega^n_{\epsilon}}{ \Omega^n_0}\right) \in \xi^{n(k-1)} \cA_{\phg}(\cU_{2,0}).
  \end{gathered}  
  \label{w.26}\end{equation}  
 Combining the diffeomorphisms $\varphi^{12}$ and $\varphi_2$ then yields the following.
   \begin{theorem}
 There is an open neighborhood $\cV$ of $\cH_2$ in $\cX$ and a polyhomogeneous map $\psi_{\epsilon}: \cW_{0,k}\cap \cV\to \cW_{\epsilon,k}$ which is a homeomorphism and, when restricted to $W_{0,k}\cap \cV$,  a diffeomorphism  such that
\begin{equation}
 \begin{gathered}
  \psi_{\epsilon}^* J_{\epsilon}- J_0\in x_2^{\frac{k(n-1)}{n-k}} \cA_{\phg}(\cW_{0,k}\cap \cV; \End({}^{w}T(W_{0,k}\cap \cV))),  \\
  \log\left( \frac{\psi_{\epsilon}^* \Omega^n_{\epsilon}} {\Omega^n_0}\right) \in x_2^{\frac{k(n-1)}{n-k}} \cA_{\phg}(\cW_{0,k}\cap \cV).
  \end{gathered}  
  \label{w.28}\end{equation}   
\label{w.27}\end{theorem} 

\begin{corollary}
The potential $u_{\epsilon,k}$ of \eqref{w.16} can be extended to a function $u_{\epsilon,k}\in x^{-\frac{2k(n-1)}{n(k-1)}}\cA_{\phg}(\cW_{\epsilon,k})$ in such a way that $\frac{\sqrt{-1}}2\pa \db u_{\epsilon,k}$ is the Kähler form of a warped $\QAC$-metric with 
\begin{equation}
      \psi_{\epsilon}^*\left( \frac{\sqrt{-1}}2 \pa\db u_{\epsilon,k}\right)- \frac{\sqrt{-1}}2\pa\db u_{0,k}\in x_2^{\mu}\cA_{\phg}(\cW_{0,k}\cap\cV;\Lambda^2({}^{w}T^*\cW_{\epsilon,k})) 
\label{w.29}\end{equation}
and 
\begin{equation}
    \log \left( \frac{(\frac{\sqrt{-1}}2\pa\db u_{\epsilon,k})^n}{c_n \Omega^n_{\epsilon}\wedge \overline{\Omega}^n_{\epsilon}} \right) \in x_1 x_2^{\mu} \cA_{\phg}(\cW_{\epsilon,k}).  
\label{w.30}\end{equation}
\label{w.31}\end{corollary}
\begin{proof}
By Theorem~\ref{ac.11}, we have that
$$
     (\varphi_{\epsilon}^{12})^* u_{\epsilon,k}- u_{0,k}\in x_2^{\mu-2}\cA_{\phg}(\cW_{0,k}\cap \cV).
$$
Thus, we can extend $u_{\epsilon,k}$ to a smooth function on $W_{\epsilon,k}$ such that 
$$
    \psi_{\epsilon}^* u_{\epsilon,k} -u_{0,k}\in x_2^{\mu-2}\cA_{\phg}(\cW_{0,k}\cap\cV).
$$
Moreover, by Theorem~\ref{w.27}, we have that \eqref{w.29} holds.  Hence, combining this with Lemma~\ref{w.21} gives \eqref{w.30}.  In particular, by Lemma~\ref{w.17} and \eqref{w.29}, we have that 
$$
    \frac{\sqrt{-1}}{2}\pa \db u_{\epsilon}>0 
$$
outside a compact set of $W_{\epsilon,k}$.  In other words, there is a constant $C>0$ such that $\frac{\sqrt{-1}}{2} \pa\db u_{\epsilon,k} >0$ at points $p\in W_{\epsilon,k}$ where 
$u_{\epsilon,k}(p)>C$.  To see that we can choose $u_{\epsilon,k}$ such that $\frac{\sqrt{-1}}{2}\pa \db u_{\epsilon}>0$ everywhere on $W_{\epsilon,k}$, let $\eta\in\CI(\bbR)$ be a non-decreasing convex function such that
$$
    \eta(t) = \left\{ \begin{array}{ll} t, & \mbox{if} \; t\ge C+2, \\
                                                   C+\frac{3}2, & \mbox{if} \; t\le C+1. \end{array}  \right.
$$
For such a function, $\eta\circ u_{\epsilon,k}$ is equal to $u_{\epsilon,k}$ outside a compact set.  Moreover, everywhere on $W_{\epsilon,k}$, we have that
\begin{equation}
  \frac{\sqrt{-1}}2 \pa \db \eta\circ u_{\epsilon,k} = \frac{\sqrt{-1}}2 \eta''(u_{\epsilon,k}) \pa u_{\epsilon,k}\wedge \db u_{\epsilon,k} + \frac{\sqrt{-1}}2 \eta'(u_{\epsilon,k}) \pa\db u_{\epsilon,k}\ge
   \frac{\sqrt{-1}}2 \eta'(u_{\epsilon,k}) \pa\db u_{\epsilon,k}\ge 0,
\label{w.32}\end{equation}
since $\eta', \eta''\ge 0$.  Now, let $w_{\epsilon,k}$ be the restriction to $W_{\epsilon,k}$ of the potential 
$$
   \sum_{i=1}^{n+1} |z_i|^2
$$
of the Euclidean metric on $\bbC^{n+1}$ and let $\phi\in \CI_{c}(W_{\epsilon,k})$ be a smooth function such that 
$$
   \phi(p)= \left\{ \begin{array}{ll} 1, & \mbox{if} \; u_{\epsilon,k}(p)\le C+2, \\
                                                   0, & \mbox{if} \; u_{\epsilon,k}(p)\ge C+3. \end{array}  \right.
$$
Then for $\delta>0$,
\begin{equation}
   \widetilde{u}_{\epsilon,k}:=\eta\circ u_{\epsilon,k}+ \delta\phi w_{\epsilon,k}
\label{w.33}\end{equation}
is equal to $u_{\epsilon,k}$ outside a compact set.  Moreover, thanks to \eqref{w.32}, at a point $p$ where $u_{\epsilon,k}(p)< C+2$, we have that
$$
          \frac{\sqrt{-1}}2\pa\db \widetilde{u}_{\epsilon,k} \ge \delta \frac{\sqrt{-1}}2\pa\db w_{\epsilon,k}>0,
$$
while at points where $u_{\epsilon,k}(p)>C+3$,
$$
 \frac{\sqrt{-1}}2\pa\db \widetilde{u}_{\epsilon,k} = \frac{\sqrt{-1}}2\pa\db u_{\epsilon,k} >0.
$$
On the other hand, in the compact region where $C+2\le u_{\epsilon,k}(p)\le C+3$, we have that $\frac{\sqrt{-1}}2 \pa\db \eta\circ u_{\epsilon,k}= \frac{\sqrt{-1}}2\pa\db u_{\epsilon,k}>0$, so taking $\delta>0$ sufficiently small, we can also ensure that $\frac{\sqrt{-1}}2\pa\db \widetilde{u}_{\epsilon,k}>0$ there too, and hence everywhere on $W_{\epsilon,k}$.  Thus, replacing $u_{\epsilon,k}$ with $\widetilde{u}_{\epsilon,k}$ yields the desired potential.    
\end{proof}

The previous corollary gives us a Kähler warped $\QAC$-metric with Kähler form $\frac{\sqrt{-1}}2\pa\db u_{\epsilon,k}$ and with Ricci potential 
\begin{equation}
 r_{\epsilon,k}:=\log \left( \frac{(\frac{\sqrt{-1}}2\pa\db u_{\epsilon,k})^n}{c_n \Omega^n_{\epsilon}\wedge \overline{\Omega}^n_{\epsilon}} \right) \in x_1 x_2^{\mu} \cA_{\phg}(\cW_{\epsilon,k})
 \label{w.34}\end{equation}
decaying at infinity.  To obtain a Calabi-Yau warped $\QAC$-metric, it suffices then to solve the complex Monge-Ampère equation
\begin{equation}
 \log \left( \frac{(\frac{\sqrt{-1}}2\pa\db(u_{\epsilon,k}+u))^n}{c_n\Omega_{\epsilon}^n\wedge\overline{\Omega}^n_{\epsilon}} \right)=-r_{\epsilon,k}.
\label{w.35}\end{equation}
In order to do this, we need first to establish some mapping properties for the Laplacian associated to a warped $\QAC$-metric.

\section{Mapping properties of the Laplacian} \label{mp.0}

The result of \cite{DM2014} about the mapping properties of the Laplacian of a $\QAC$-metric has a direct analogue for warped $\QAC$-metrics.  The strategy of the proof is the same as in \cite{DM2014}, although some of the details are slightly different.  For this reason, we will recall the general strategy of \cite{DM2014}, but we will mostly focus on the details that differ from \cite{DM2014}.     

First, recall from \cite{DM2014} that if $g$ is a complete Riemannian metric on a manifold $Z$ and $h$ is a positive smooth function on $Z$, then one can consider the measure $d\mu= h^2 dg$, where $dg$ is the volume density of $g$.  In the terminology of \cite{GS2005}, the triple $(Z,g,\mu)$ is a \textbf{complete weighted Riemannian manifold}.  On such a manifold, we use the Riemannian metric $g$ to define the distance $d(p,q)$ between two points of $p,q\in Z$, and we use the notation
\begin{equation}
    B(p,r):= \{ q\in Z \; | \; d(p,q)<r\}
\label{mp.1}\end{equation}
to denote the geodesic ball of radius $r$ centred at $p\in Z$.  However, to measure the volume of such a  ball, we use the weighted measure $\mu$ instead of the volume density of $g$.  Moreover, the $L^2$-inner product on functions we consider is the one induced by $\mu$, namely
\begin{equation}
  \langle u, v\rangle_{\mu}:= \int_Z  uv d\mu.
\label{mp.2}\end{equation}
Let $\nabla$ be the Levi-Civita connection of $g$ and $\Delta=-\operatorname{div}\circ \nabla$ the Laplacian of $g$.  For $\cR$ a function, we want to consider the operator
$\cL:= \Delta+\cR$ and its \textbf{Doob transform} with respect to $h$, 
\begin{equation}
    \mathcal{\cL}:=  h^{-1}\circ \cL \circ h = \Delta_\mu +V +\cR,
\label{mp.3}\end{equation}
where $V:= -\frac{\Delta h}{h}$ and $\Delta_{\mu}= \nabla^{*,\mu}\nabla$ with $\nabla^{*,\mu}$ the adjoint of $\nabla$ with respect to the $L^2$-inner product \eqref{mp.2} and the $L^2$-inner product on forms given by 
$$
        \langle \eta_1,\eta_2\rangle_{\mu}:= \int_Z (\eta_1(z),\eta_2(z))_{g(z)} d\mu(z).
$$ 
Let $H_{\cL}(t,z,z')$ and $H_{\Delta+V}(t,z,z')$ denote the heat kernels of $\cL$ and $\Delta+V$ with respect to the volume density $dg$,
$$
           (e^{-t\cL}u)(z)= \int_Z H_{\cL}(t,z,z') u(z') dg(z'), \quad  (e^{-t( \Delta+V)}u)(z)= \int_Z H_{\Delta+V}(t,z,z') u(z') dg(z'),
$$
and let 
$$
     G_{\cL}(z,z')= \int_0^{\infty} H_{\cL}(t,z,z')dt \quad \mbox{and}\quad G_{\Delta+V}(z,z')= \int_0^{\infty} H_{\Delta+V}(t,z,z')dt
$$
be the corresponding Green's functions.  
Let also $H_{\Delta_{\mu}}(t,z,z')$ be the heat kernel of $\Delta_{\mu}$ with respect to the measure $\mu$,
$$
  (e^{-t\Delta_{\mu}}u)(xz)= \int_Z H_{\Delta_{\mu}}(t,z,z') u(z') d\mu(z')
$$
with corresponding Green's function 
$$
  G_{\Delta_{\mu}}(z,z')= \int_0^{\infty} H_{\Delta_{\mu}}(t,z,z')dt.
$$
These heat kernels and Green's functions are related as follows.  
\begin{lemma}[Theorem~3.12 in \cite{DM2014}]
If $\cR\ge V$, then 
$$
       | H_{\cL}(t,z,z')|\le H_{\Delta+V}(t,z,z')= h(z) h(z') H_{\Delta_{\mu}}(t,z,z'),
$$
and hence
$$
      |G_{\cL}(z,z')|\le G_{\Delta+V}(z,z')= h(z)h(z')G_{\Delta_{\mu}}.
$$
\label{mp.4}\end{lemma}
Now, to bound $H_{\Delta_{\mu}}$, one can use the method of Grigor'yan and Saloff-Coste involving the following notions.
\begin{definition}
The complete weighted Riemannian manifold $(Z,g,\mu)$ satisfies 
\begin{itemize}
\item[$(VD)_{\mu}$] the weighted volume doubling property if there exists $C_D>0$ such that 
$$
      \mu(B(p,2r))\le C_D \mu(B(p,r))
$$
for all $p\in Z$ and all $r>0$;
\item[$(PI)_{\mu,\delta}$] the uniform Poincaré inequality with parameter $\delta\in (0,1]$ if there exists a constant $C_{P}>0$ such that
$$
       \int_{B(p,r)} (f-\overline{f})^2d\mu \le C_P r^2 \int_{B(p,\delta^{-1}r)} |d f|_g^2 d\mu
$$  
for all $f\in W^{1,2}_{loc}(Z)$, all $p\in Z$ and all $r>0$;
\item[$(PI)_{\mu}$] the uniform weighted Poincaré inequality if we can take $\delta=1$ in the previous statement.
\end{itemize}
\label{mp.5}\end{definition}

\begin{theorem}[Theorem~2.7 in \cite{GS2005}]
Let $(Z,g,\mu)$ be a complete weighted Riemannian manifold satisfying $(VD)_{\mu}$ and $(PI)_{\mu}$.  Then 
$$
      H_{\Delta_{\mu}}(t,z,z')\asymp \left( \mu(B(z,\sqrt{t})) \mu(B(z',\sqrt{t})) \right)^{-\frac12} e^{-c \frac{d(z,z')^2}{t}}
$$
for all $(t,z,z')\in (0,\infty)\times Z\times Z$.
\label{mp.6}\end{theorem}
\begin{corollary}
If the complete weighted Riemannian manifold $(Z,g,\mu)$ satisfies $(VD)_{\mu}$ and $(PI)_{\mu}$, then 
$$
       H_{\Delta+V}(t,z,z') \asymp h(z)h(z')\left( \mu(B(z,\sqrt{t})) \mu(B(z',\sqrt{t})) \right)^{-\frac12} e^{-c \frac{d(z,z')^2}{t}}.
$$ 
\label{mp.7}\end{corollary}

We want to apply this result to the case where $Z= W_{\epsilon,k}$ with $\epsilon\ne 0$, $g$ is a warped $\QAC$-metric and $d\mu= d\mu_{a,b}= \rho^a w^b dg$.  For this purpose, we need to know that $(VD)_{\mu}$ and $(PI)_{\mu}$ hold for this particular complete weighted Riemannian manifold for suitable choices of $a$ and $b$.  To check this, we  introduce the following notation.
\begin{definition}
Declare $0\in W_{\epsilon,k}\subset \bbC^{n+1}$ to be the basepoint of $W_{\epsilon,k}$.  A ball centred at $0$ is called \textbf{anchored} and we denote its volume by
$$
   \cA(R;a,b):= \mu_{a,b}(B(0,R)).
$$
Fix $c\in(0,1)$.  With respect to this choice, a ball $B(p,r)$ is said to be \textbf{remote} if $r<c d(0,p)$, in which case we use the notation
$$
         \cR(p,r;a,b):= \mu_{a,b}(B(p,r)).
$$
If $B(p,r)$ is any ball, possibly neither anchored nor remote, we use the notation
$$
    \cV(p,r;a,b):= \mu_{a,b}(B(p,r)).
$$
\label{mp.8}\end{definition}

\begin{proposition}
For $R>1$, $a\ne -2n$ and $b\ne -2n+2$, we have that 
$$
     \cA(R;a,b)\asymp \left\{\begin{array}{ll}1+R^{a+2n}+ R^{\frac{2n}k+a -b\left( \frac{n(k-1)}{k(n-1)} \right)}, & \mbox{if} \; \frac{2n}k+a -b\left( \frac{n(k-1)}{k(n-1)}\right)\ne 0, \\
        1+R^{a+2n}+ \log R, & \mbox{if} \; \frac{2n}k+a -b\left( \frac{n(k-1)}{k(n-1)}\right)= 0.
        \end{array}  \right.
$$
\label{mp.9}\end{proposition}
\begin{proof}
Consider the two regions
\begin{equation}
V_1:= \left\{ p\in W_{\epsilon,k} \; | \; w(p) < 1-c\right\} \quad \mbox{and} \quad
V_2:= \left\{ p\in W_{\epsilon,k} \; | \; w(p) \ge 1-c\right\}. 
\label{mp.10}\end{equation}
In the region $V_2$, the metric $g$ behaves like an $\AC$-metric.  Moreover, since by definition $w\ge 1-c$ on $V_2$, we compute as in \cite{DM2014} that 
\begin{equation}
  \mu_{a,b}(B(0,R)\cap V_2)\asymp 1 + R^{a+2n} \quad \mbox{for} \quad R>1.
\label{mp.11}\end{equation}
In the region $V_1$, using the coordinates $\zeta, z_{n+1}$,  take $\rho= |z_{n+1}|$, $w=\frac{\langle \zeta\rangle^{\frac{n-k}{n-1}}}{\rho^{\frac{n(k-1)}{k(n-1)}}}$ and suppose that $g$ is of the form
\begin{equation}
  d\rho^2 + \rho^2 d\theta^2 + \rho^{\frac{2(n-k)}{k(n-1)}}g_{V_{\epsilon,k}}.
\label{mp.12}\end{equation}
If we set $\rho_2:= \langle \zeta  \rangle^{\frac{n-k}{n-1}}$, then we see that the volume density of $g$ is given by
\begin{equation}
 dg= \rho^{\frac{2(n-k)}{k}+1} d\rho d\theta dg_{V_{\epsilon,k}} \quad \Longrightarrow \quad d\mu_{a,b}= \rho^aw^b dg= \rho^{\frac{2(n-k)}k+1+a -b\left( \frac{n(k-1)}{k(n-1)} \right)}
 d\rho d\theta \rho_2^b dg_{V_{\epsilon,k}}.
\label{mp.13}\end{equation}
Notice also that $w<1-c$ if and only if $\rho_2<\rho^{\frac{n(k-1)}{k(n-1)}}(1-c)$ and that outside a compact set of $V_{\epsilon,k}$, we have that 
$$
        \rho_2^bdg_{V_{\epsilon,k}} \asymp \rho_2^b \rho_2^{2n-3}d\rho_2 d\kappa,
$$
where $d\kappa$ is a volume density on $\pa\bV_{\epsilon,k}$.  Hence, we compute that
\begin{equation}
\begin{aligned}
 \mu_{a,b}(B(0,R)\cap V_1) & \asymp 1+ \int_1^R  \rho^{\frac{2(n-k)}k+1+a -b\left( \frac{n(k-1)}{k(n-1)} \right)} \left(  \int_{\left\{\rho_2<(1-c)\rho^{\frac{n(k-1)}{k(n-1)}}   \right\}}\rho_2^b dg_{V_{\epsilon,k}}\right)d\rho \\
  &\asymp 1 + \int_1^R \rho^{\frac{2(n-k)}k+1+a -b\left( \frac{n(k-1)}{k(n-1)} \right)} \left( 1+ ((1-c)\rho^{\frac{n(k-1)}{k(n-1)}})^{b+2n-2} \right)d\rho \\
  & \asymp \left\{\begin{array}{ll}1+R^{a+2n}+ R^{\frac{2n}k+a -b\left( \frac{n(k-1)}{k(n-1)} \right)}, & \mbox{if} \; \frac{2n}k+a -b\left( \frac{n(k-1)}{k(n-1)}\right)\ne 0, \\
        1+R^{a+2n}+ \log R, & \mbox{if} \; \frac{2n}k+a -b\left( \frac{n(k-1)}{k(n-1)}\right)= 0,
        \end{array}  \right.
 \end{aligned}
 \label{mp.14}\end{equation}
 from which the result follows.  
\end{proof}
\begin{corollary}
If $a$ and $b$ satisfy 
\begin{equation}
   a+2n> 0 \quad \mbox{and}  \quad b+2n-2> 0, 
\label{mp.15}\end{equation}
then 
$$
       \cA(R;a,b)\asymp R^{a+2n} \quad \mbox{for} \; R>1.
$$
\label{mp.16}\end{corollary}
\begin{proof}
If $-b< 2n-2$, then for $R>1$, 
$$
      R^{\frac{2n}k+a -b\left( \frac{n(k-1)}{k(n-1)} \right)}\le R^{\frac{2n}k+a + \frac{2n(k-1)}{k}}= R^{2n+a} \quad \mbox{for}  \quad  \frac{2n}k+a -b\left( \frac{n(k-1)}{k(n-1)}\right)\ne 0,
$$
whereas  if $ \frac{2n}k+a -b\left( \frac{n(k-1)}{k(n-1)}\right)=0$, we have instead that
$$
      \log R< CR^{\frac{2n}k+a + \frac{2n(k-1)}{k}}= CR^{2n+a}
$$
for some constant $C>0$.  The result follows by combining these observations with Proposition~\ref{mp.9}.
\end{proof}
\begin{remark}
Notice that Proposition~\ref{mp.9} is slightly different than the corresponding result for $\QAC$-metrics, but that Corollary~\ref{mp.16} is the exact analogue of the corresponding statement for $\QAC$-metrics, \cf \cite[Proposition~4.9 and (4.11)]{DM2014}.
\end{remark}

We  can also estimate the volume of remote balls.  Since $\rho\asymp d(0,\cdot)$ outside a compact set, we can use the function $\rho$ instead of $d(0,\cdot)$ to define remote balls.  That is, we will say that $B(p,r)$ is remote if  $r\in (0, c\rho(p))$.  
\begin{proposition}
Suppose that we have chosen the remote parameter $c\in(0,1)$ so that in fact $c\in(0,\frac13)$, that $a\ne -2n$ and that  $b\ne -2n+2$.  Fix $p\in W_{\epsilon,k}$.  If $w(p)\ge 1-3c$, then for $r\in (0,c\rho(p))$,
$$
     \cR(p,r;a,b) \asymp \rho^a r^{2n}.
$$
If instead $w(p)< 1-3c$, then
$$
    \cR(p,r;a,b)\asymp \left\{  \begin{array}{ll} w(p)^b \rho(p)^a r^{2n}, & \mbox{if} \; r<cw(p)\rho(p), \\
                                                                      \rho(p)^{a-b} r^{2n+b} & \mbox{if} \; r\ge cw(p) \rho(p).  \end{array} \right.
$$ 
\label{mp.17}\end{proposition}
\begin{proof}
First notice that $\rho(z)\asymp \rho(p)$ for $z\in B(p,c\rho(p))$.  Now, if $w(p)\ge 1-3c$, we are in a region where $g$ behaves like an $\AC$-metric, so the result follows as in \cite[Proposition~4.13]{DM2014}.  In fact, for $\AC$-metrics, recall from \cite[Proposition~{4.15}]{DM2014} that, using a simple rescaling argument, we have also the following estimate for the volume of non-remote balls,
\begin{equation}
  \int_{B_{g_{V_{\epsilon,k}}}(p,r)} \rho_2^b dg_{V_{\epsilon,k}}  \asymp r^{b+ 2n-2} \quad \mbox{if} \; r\ge c\rho_2(p),  
\label{mp.18}\end{equation}
where $B_{g_{V_{\epsilon,k}}}(p,r)$ is the geodesic ball of radius $r$ centred at $p\in V_{\epsilon,k}$ with respect to the $\AC$-metric $g_{V_{\epsilon,k}}$.

If instead we have that $w(p)<1-3c$, then notice that $B(p,r)\subset V_1$ with $V_1$ as in \eqref{mp.10}.  Indeed, since $w=\frac{r_2}{\rho}$ with $r_2:= \rho^{\frac{n-k}{k(n-1)}}\langle \zeta\rangle^{\frac{n-k}{n-1}}$, we see that if $q\in B(p, r)$, then
\begin{equation}
  \rho(q)\ge \rho(p) -d(p,q)\ge \rho(p) -r\ge \rho(p) -c\rho(p) = (1-c)\rho(p),
\label{mp.18c}\end{equation}
and
$$
r_2(q)\le r_2(p) + d(p,q)\le r_2(p)+ c\rho(p),
$$
so that 
\begin{equation}
 w(q) \le \frac{r_2(p)+ c\rho(p)}{\rho(q)}< \frac{(1-3c)\rho(p)+ c\rho(p)}{\rho(q)}\le \frac{1-2c}{1-c}= 1 -\frac{c}{1-c}< 1-c.
\label{mp.18b}\end{equation}
We can thus assume that $g$ is given by \eqref{mp.12} in the decomposition $\bbC\times V_{\epsilon,k}$.  If $p$ corresponds to the point $(p_1,p_2)\in \bbC\times V_{\epsilon,k}$ and if $B_1(p_1, r)$ denotes the geodesic ball in $\bbC$ with respect to the Euclidean metric $d\rho^2+ \rho^2 d\theta^2$ and $B_2(p_2, r)$ denotes the geodesic ball in $V_{\epsilon,k}$ with respect to the metric $g_{V_{\epsilon,k}}$, then recalling that $\rho\asymp \rho(p)$ on $B(p,r)$, notice that the weighted volume of $B(p,r)$ is comparable to the one of the product of balls
$$
        B_1(p_1,r) \times B_2\left(p_2, \frac{r}{\rho(p)^{\frac{n-k}{k(n-1)}}}\right).
$$
Now, we have that
\begin{equation}
\mu_{a,b}\left( B_1(p_1,r) \times B_2\left(p_2, \frac{r}{\rho(p)^{\frac{n-k}{k(n-1)}}}\right)\right) = \int_{B_1(p,r)} \left( \int_{B_2\left(p_2, \frac{r}{\rho(p)^{\frac{n-k}{k(n-1)}}}\right)} \rho_2^b dg_{V_{\epsilon,k}} \right) \rho^{\frac{2(n-k)}k +1+a -b\left( \frac{n(k-1)}{k(n-1)} \right)}d\rho d\theta.
\label{mp.19}\end{equation}
Notice that $r\ge c\rho(p)w(p)$ if and only if $\frac{r}{\rho(p)^{\frac{n-k}{k(n-1)}}}\ge c\rho_2(p)$,  so using the fact that $g_{V_{\epsilon,k}}$ is an $\AC$-metric, we see that
$$
\int_{B_2\left(p_2, \frac{r}{\rho(p)^{\frac{n-k}{k(n-1)}}}\right)} \rho_2^b dg_{V_{\epsilon,k}} \asymp \left\{  \begin{array}{ll} \rho_2(p)^b \left( \frac{r}{\rho(p)^{\frac{n-k}{k(n-1)}}} \right)^{2n-2}, & \mbox{if} \; r<cw(p)\rho(p), \\
\left( \frac{r}{\rho(p)^{\frac{n-k}{k(n-1)}}} \right)^{2n-2+b},& \mbox{if} \; r\ge cw(p) \rho(p).  \end{array} \right.
$$
Substituting this in \eqref{mp.19} and using the fact that the Euclidean metric on $\bbC$ is an $\AC$-metric, we find that
$$
\mu_{a,b}\left( B_1(p_1,r) \times B_2\left(p_2, \frac{r}{\rho(p)^{\frac{n-k}{k(n-1)}}}\right)\right)\asymp \left(\frac{\rho_2(p)}{\rho(p)^{\frac{n(k-1)}{k(n-1)}}}\right)^{b}\rho(p)^a r^{2n}= w(p)^b \rho(p)^a r^{2n}
$$
if $r< cw(p) \rho(p)$, and 
$$
\mu_{a,b}\left( B_1(p_1,r) \times B_2\left(p_2, \frac{r}{\rho(p)^{\frac{n-k}{k(n-1)}}}\right)\right)\asymp \rho(p)^{a-b}r^{2n+b} 
$$
if $r\ge cw(p) \rho(p)$, from which the result follows.  
 
\end{proof}

\begin{corollary}
For $a$ and $b$ as in Proposition~\ref{mp.17}, remote balls satisfy $(VD)_\mu$.  
\label{mp.19b}\end{corollary}

\begin{corollary}
If $a$ and $b$ are as in Proposition~\ref{mp.17}, then for any $p\in W_{\epsilon,k}$ and for $r=c\rho(p)$, we have that 
$$
   \cR(p,r;a,b)\asymp r^{2n+a}\asymp \rho(p)^{2n+a}.
$$
\label{mp.20}\end{corollary}
\begin{proof}
If $w(p)\ge 1-3c$, this is a direct consequence of the previous proposition.  If instead $w(p) < 1-3c$, then notice that 
$$
     r= c\rho(p)\ge c\rho(p) (1-3c)> c\rho(p) w(p),
$$ 
so the result is again a consequence of Proposition~\ref{mp.17}.
\end{proof}
We can also use this result to estimate the volume of non-remote balls.
\begin{corollary}
Suppose that $a$ and $b$ satisfy \eqref{mp.15}.  Then for any $p\in W_{\epsilon,k}$ and any $r\ge c\rho(p)$, we have that
$$
         \cV(p,r;a,b)\asymp r^{2n+a}.
$$
\label{mp.20b}\end{corollary}
\begin{proof}
Without loss of generality, we can assume that $\min \rho\ge 1$ and that $\rho(p)\ge d(0,p)$ for all $p\in W_{\epsilon,k}$.  Given $p\in W_{\epsilon,k}$, suppose first that $r\ge 3\rho(p)$.  In this case,
$$
    B(0,\frac{2r}3)\subset B(p,r) \subset B(0,\frac{4r}3), 
$$
so that 
$$
   \cA(\frac{2r}3;a,b) \le \cV(p,r;a,b)\le \cA(\frac{4r}3;a,b)
$$
and the result follows from Corollary~\ref{mp.16}.  If instead $c\rho(p)\le r< 3\rho(p)$, then
$$
     B(p,c\rho(p))\subset B(p,r)\subset B(0,4\rho(p)),
$$
so that 
$$
   \cR(p,c\rho(p);a,b)\le \cV(p,r;a,b)\le \cA(4\rho(p);a,b)
$$
and we deduce from Corollary~\ref{mp.16} and Corollary~\ref{mp.20} that 
$$
     \cV(p,r;a,b)\asymp \rho(p)^{2n+a}\asymp r^{2n+a}.
$$
\end{proof}

\begin{corollary}
Suppose that $a$ and $b$ satisfy condition \eqref{mp.15}.  Then there exists a constant $C>0$ such that 
\begin{equation}
    \cA(\rho(p);a,b)\le C \cR(p,c\rho(p);a,b)      
\label{mp.21}\end{equation}
for any $p\in W_{\epsilon,k}$.
\label{mp.22}\end{corollary}
\begin{proof}
It suffices to combine Corollary~\ref{mp.16} with Corollary~\ref{mp.20}.
\end{proof}

\begin{corollary}
Suppose that $a$ and $b$ satisfy condition \eqref{mp.15}.  If the complete weighted Riemannian manifold $(W_{\epsilon,k},g,\mu_{a,b})$ satisfies $(VD)_{\mu}$ and $(PI)_{\mu,\delta}$ with parameter $\delta\in (0,1]$  for all remote balls, then $(VD)_{\mu}$ and $(PI)_{\mu,\delta}$ hold for all balls of $(W_{\epsilon,k},g,\mu_{a,b})$.
\label{mp.23}\end{corollary}
\begin{proof}
By \cite[Theorem~5.2]{GS2005} and Corollary~\ref{mp.22}, it suffices to check that $(W_{\epsilon,k},g)$ satisfies the property of relatively connected annuli (RCA) with respect to $0$, that is, that there exists $C_A>1$ such that for all $r>C_A^2$, for all $p,q\in W_{\epsilon,k}$ with $d(0,p)=d(0,q)=r$, there exists a continuous path $\gamma: [0,1]\to W_{\epsilon,k}$ with $\gamma(0)=p$, $\gamma(1)=q$ and with image contained in the annulus $B(0,C_Ar)\setminus B(0, C_A^{-1}r)$.  But clearly the RCA property holds on $(W_{\epsilon,k},g)$, so the result follows.  
\end{proof}
Thus, by Corollary~\ref{mp.23} and Corollary~\ref{mp.19b}, to show that $(VD)_{\mu}$ and $(PI)_{\mu,\delta}$ hold for the complete weighted Riemannian manifold $(W_{\epsilon,k},g,\mu_{a,b})$ with $a$ and $b$ satisfying \eqref{mp.15}, it remains to check that $(PI)_{\mu,\delta}$ holds on remote balls.

\begin{theorem}
A warped $\QAC$-metric $g$ on $W_{\epsilon,k}$ is such that for $a$ and $b$ satisfying \eqref{mp.15}, the properties $(VD)_{\mu}$ and $(PI)_{\mu}$ hold on $(W_{\epsilon,k},g,\mu_{a,b})$.  
\label{mp.24}\end{theorem}
\begin{proof}
By the previous results, it suffices to check that for some $\delta\in (0,1)$, $(PI)_{\mu,\delta}$ holds on remote balls.  Indeed, in that case, $(VD)_{\mu}$ and $(PI)_{\mu,\delta}$ will hold for all balls by Corollary~\ref{mp.23}, so by the argument of Jerison \cite{Jerison1986}, $(PI)_{\mu}$ will also hold for all balls.  We also choose our remote parameter $c$ to be in the interval $(0,\frac14)$.  Thus, let $B(p,r)$ be a remote ball.  If $w(p)\ge 1-4c$, then we are in a region where $g$ behaves like an $\AC$-metric, so we can apply the rescaling argument of \cite[Proposition~4.20]{DM2014} to conclude that $(PI)_{\mu}$ holds on $B(p,r)$.  If instead $w(p)<1-4c$, then as discussed in \eqref{mp.18b}, we have that $B(p,r)\subset V_1$ and we can suppose that $g$ is of the form \eqref{mp.12}.  Regarding $V_1$ as a subset of $\bbC\times V_{\epsilon,k}$ with $p$ corresponding to the point $(p_1,p_2)\in \bbC\times V_{\epsilon,k}$, notice that $B(p,r)$ is contained in the product of balls
\begin{equation}
     Q(r)= B_1\times B_2:= B_1(p_1,r) \times B_2\left(p_2, \frac{r}{((1-c)\rho(p))^{\frac{n-k}{k(n-1)}}}\right).
\label{mp.25}\end{equation}
Before proving the uniform Poincaré inequality on $B(p,r)$, let us prove it for $Q(r)$.  To do this, write $d\mu_{a,b}= d\mu_1 d\mu_2$ with
$$
    d\mu_1= \rho^{\frac{2(n-k)}k+a-b\left( \frac{n(k-1)}{k(n-1)} \right)} \rho d\rho d\theta \quad \mbox{and} \quad d\mu_2= \rho_2^b dg_{V_{\epsilon,k}}.
$$
Given a function $f$ on $Q(r)$, we define the partial averages 
$$
      \barf_i:= \frac{1}{\mu_i(B_i)} \int_{B_i} f d\mu_i,  \quad i\in \{1,2\}, \quad \barf_Q:= \frac{1}{\mu_{a,b}(Q(r))}\int_{Q(r)} fd\mu_{a,b}.
$$
Notice in particular that $\barf_Q= \overline{(\barf_1)}_2= \overline{(\barf_2)}_1$.  Now, since $g_{V_{\epsilon,k}}$ is an $\AC$-metric, we know from \cite{DM2014} that $(PI)_{\mu_2}$ holds on 
$(V_{\epsilon,k},g_{V_{\epsilon,k}}, \mu_2)$.  Hence, we compute that
\begin{equation}
\begin{aligned}
  \int_{Q(r)} |f-\barf_Q|^2 d\mu_1 d\mu_2  & \le 2\int_{Q(r)} \left( |f-\barf_2|^2+ |\barf_2-\barf_Q|^2 \right)d\mu_1 d\mu_2 \\
    &\le  2 \int_{B_1}  \left( C_2 \left( \frac{r}{((1-c)\rho(p))^{\frac{n-k}{k(n-1)}}} \right)^2\int_{B_2} |d_2 f|^2_{g_{V_{\epsilon,k}}}  d\mu_2 + \int_{B_2} |\barf_2-\barf_Q|^2 d\mu_2 \right) d\mu_1 \\
    &= 2 \int_{B_1} \left( C_2 r^2\int_{B_2} |d_2 f|^2_{((1-c)\rho(p))^{\frac{n-k}{k(n-1)}}g_{V_{\epsilon,k}}}  d\mu_2 + \int_{B_2} |\barf_2-\barf_Q|^2 d\mu_2 \right) d\mu_1,
\end{aligned}    
\label{mp.26}\end{equation}
where $d_i$ corresponds to the exterior differential taken only on the factor $B_i$ and $C_2>1$ is the constant for the property $(PI)_{\mu_2}$ on  $(V_{\epsilon,k},g_{V_{\epsilon,k}}, \mu_2)$.  Now, since $(1-c)\rho(p)\le \rho(q) \le (1+c)\rho(p)$ for $q\in Q(r)$,  the first term on the right hand side of \eqref{mp.26} is bounded by 
$$
2 C_2 \left( \frac{1+c}{1-c} \right)^{\frac{2(n-k)}{k(n-1)}} r^2\int_Q |df|^2_g d\mu_{a,b}.
$$
For the second term, we can instead apply $(PI)_{\mu_1}$ on $B_1$,
\begin{equation}
    \int_{Q(r)} |\barf_2-\barf_Q|^2 d\mu_1 d\mu_2 \le \int_{B_2} \left( C_1r^2 \int_{B_1} |d\barf_2|^2_{g_1} d\mu_1 \right) d\mu_2.
\label{mp.27}\end{equation}
Since 
$$
d_1\barf_2 = \frac{1}{\mu_2(B_2)} \int_{B_2} d_1f d\mu_2,
$$
we see from the Cauchy-Schwarz inequality that 
\begin{equation}
      |d_1\barf_2|^2_{g_1}\le \frac{1}{\mu_2(B_2)}\int_{B_2} |d_1f|^2_{g_1} d\mu_2,
\label{mp.28}\end{equation}
where $g_1= d\rho^2 + \rho^2 d\theta^2$ is the Euclidean metric on $\bbC$.  Thus, inserting \eqref{mp.28} in \eqref{mp.27}, we see that
\begin{equation}
\begin{aligned}
\int_{Q(r)} |\barf_2-\barf_Q|^2 d\mu_1 d\mu_2 &\le \int_{B_2} \left( C_1r^2 \int_{B_1} \left( \frac{1}{\mu_2(B_2)} \int_{B_2} |d_1f|^2_{g_1}d\mu_2 \right)d\mu_1 \right) d\mu_2 \\
       &\le C_1r^2\int_{Q(r)} |d_1f|^2_{g_1} d\mu_{a,b} \\
       &\le C_1r^2\int_{Q(r)} |d f|^2_{g} d\mu_{a,b}.
\end{aligned}
\label{mp.29}\end{equation}
Hence, this shows that $(PI)_{\mu_{a,b}}$ holds on $Q(r)$.  Now, returning  to $B(p,r)$, notice that we have the sequence of inclusions
\begin{equation}
B(p,r)\subset Q(r) \subset B\left(p, r+ \left(  \frac{1+c}{1-c}\right)^{\frac{n-k}{k(n-1)}} r\right) \subset B(p,3r)
\label{mp.30}\end{equation}
for $c\in(0,\frac14)$ sufficiently small.  Hence, we have that 
\begin{equation}
\begin{aligned}
\int_{B(p,r)} |f-\barf_{B(p,r)}|^2 d\mu_{a,b} &= \inf_a \int_{B(p,r)} |f-a|^2d\mu_{a,b} \\
  & \le \int_{B(p,r)} |f-\barf_Q|^2 d\mu_{a,b} \\
  &\le \int_{Q(r)} |f-\barf_Q|^2 d\mu_{a,b}  \\
  &\le Cr^2 \int_{Q(r)} |df|^2_g d\mu_{a,b} \quad \mbox{by} \; (PI)_{\mu_{a,b}} \; \mbox{on} \; Q(r), \\
  &\le Cr^2 \int_{B(p, \delta^{-1}r)} |df|^2_g d\mu_{a,b} \quad \mbox{with}  \; \delta=\frac13.
\end{aligned}
\label{mp.31}\end{equation}
\end{proof}
\begin{corollary}
If $g$ is a warped $\QAC$-metric on $W_{\epsilon,k}$ and $a,b$ satisfy \eqref{mp.15}, then the heat kernel $H_{\cL}$ of $\cL+\cR$ with $\cR\ge V:= -\frac{\Delta(\rho^{\frac{a}2}w^{\frac{b}2})}{\rho^{\frac{a}2}w^{\frac{b}2}}$ satisfies the estimate
$$
    |H_{\cL}(t,z,z')|\le H_{\Delta+V}(t,z,z')\asymp (\mu_{a,b}(B(z,\sqrt{t})) \mu_{a,b}(B(z',\sqrt{t})))^{-\frac12} e^{-\frac{cd(z,z')^2}t}.
$$
\label{mp.32}\end{corollary}
One useful consequence of this estimate is that the Sobolev inequality holds for warped $\QAC$-metrics.

\begin{corollary}[Sobolev inequality] If $g$ is a warped $\QAC$-metric on $W_{\epsilon,k}$, then there exists a constant $C_S>0$ such that 
\begin{equation}
     \left( \int_{W_{\epsilon,k}} |u|^{\frac{2n}{n-1}} dg\right) \le C_S \int |df|^2_g dg \quad \forall u\in \CI_c(W_{\epsilon,k}).
 \label{mp.33}\end{equation}
\label{mp.34}\end{corollary}
\begin{proof}
When $a=b=0$, we have that $V=0$, so we can take $\cR=0$ in Corollary~\ref{mp.32}.  This gives a  Gaussian bound for $H_{\Delta}$, which is well-known to be equivalent to the existence of a constant $C_S>0$ such that \eqref{mp.33} holds; see \cite{Grigoryan}.
\end{proof}

More importantly, we now have all the necessary estimates to obtain the following results for the Laplacian.
\begin{theorem}
Let $g$ be a warped $\QAC$-metric on $W_{\epsilon,k}$ and denote by $\Delta$ its corresponding Laplacian.  Let $a$ and $b$ be real numbers such that
\begin{equation}
  a+2n>2, \quad  \mbox{and} \quad b+ 2n-2>2.
\label{mp.35}\end{equation}
If $\cR\in \rho^{-2}w^{-2}\CI_{\Qb}(W_{\epsilon,k})$ is a smooth function on $W_{\epsilon,k}$ such that $\cR\ge V:= -\frac{\Delta \left( \rho^{\frac{a}2} w^{\frac{b}2}\right)}{\rho^{\frac{a}2} w^{\frac{b}2}}$, then for all $\ell\in \bbN_0$ and $\alpha\in (0,1)$, the mappings
\begin{equation}
\begin{gathered}
      \Delta+\cR:  \rho^{\delta+n}w^{\tau+n-1}H^{\ell+2}_{\Qb}(W_{\epsilon,k})\to  \rho^{\delta+n-2}w^{\tau+n-3}H^{\ell}_{\Qb}(W_{\epsilon,k}), \\    
      \Delta+ \cR: \rho^{\delta}w^{\tau} \cC^{\ell+2,\alpha}_{\Qb}(W_{\epsilon,k}) \to \rho^{\delta-2}w^{\tau-2} \cC^{\ell,\alpha}_{\Qb}(W_{\epsilon,k}), 
\end{gathered}      
\label{mp.36}\end{equation}
are isomorphisms provided that 
\begin{equation}
 2-2n-\frac{a}2<\delta < \frac{a}2 \quad \mbox{and} \quad 2-(2n-2)-\frac{b}2<\tau < \frac{b}2.
\label{mp.37}\end{equation}
\label{mp.38}\end{theorem}
\begin{proof}
We proceed exactly as in \cite{DM2014}, the idea being to use Corollary~\ref{mp.32} to show that the Green's operator defines an inverse
$$
  G_{\Delta+\cR}: \rho^{\delta+n-2}w^{\tau+n-3}H^{\ell}_{\Qb}(W_{\epsilon,k})\to \rho^{\delta+n}w^{\tau+n-1}H^{\ell+2}_{\Qb}(W_{\epsilon,k})
$$
to \eqref{mp.36} using Corollary~\ref{mp.32} and the Schur test.  In \cite{DM2014}, this is done by using estimates on the volume of remote balls and non-remote balls obtained in
\cite[Proposition~4.15 and Proposition~5.2]{DM2014}.  In our setting, the exact analogues are Corollary~\ref{mp.20b} and Proposition~\ref{mp.17}, so we can apply the argument of \cite{DM2014} verbatim.  For the maps on Hölder spaces, we can also use these estimates on the Green's operator to see that it defines an inverse
$$
G_{\Delta+\cR}: \rho^{\delta-2}w^{\tau-2}\cC^{\ell,\alpha}(W_{\epsilon,k})\to L^{\infty}(W_{\epsilon,k}).
$$
Using the Schauder estimates and the fact that a $\Qb$-metric has bounded geometry, we then see that in fact it defines an inverse
$$
G_{\Delta+\cR}: \rho^{\delta-2}w^{\tau-2}\cC^{\ell,\alpha}(W_{\epsilon,k})\to \rho^{\delta}w^{\tau}\cC^{\ell+2,\alpha}(W_{\epsilon,k}).$$
\end{proof}
We will be mostly interested in the following special case of Theorem~\ref{mp.38}.
\begin{corollary}
Let $g$ be a warped $\QAC$-metric on $W_{\epsilon,k}$ and denote by $\Delta$ its corresponding Laplacian. Then for all $\ell\in \bbN_0$ and $\alpha\in (0,1)$, the mappings
\begin{equation}
\begin{gathered}
      \Delta:  \rho^{\delta+n}w^{\tau+n-1}H^{\ell+2}_{\Qb}(W_{\epsilon,k})\to  \rho^{\delta+n-2}w^{\tau+n-3}H^{\ell}_{\Qb}(W_{\epsilon,k}), \\    
      \Delta: \rho^{\delta}w^{\tau} \cC^{\ell+2,\alpha}_{\Qb}(W_{\epsilon,k}) \to \rho^{\delta-2}w^{\tau-2} \cC^{\ell,\alpha}_{\Qb}(W_{\epsilon,k}), 
\end{gathered}      
\label{mp.39}\end{equation}
are isomorphisms provided that
\begin{equation}
 2-2n<\delta < 0 \quad \mbox{and} \quad 2-(2n-2)<\tau < 0.
\label{mp.40}\end{equation}
\label{mp.41}\end{corollary}
\begin{proof}
It suffices to take $a=b=0$ and $\cR=0$ in Theorem~\ref{mp.38}.
\end{proof}

\section{Examples of Calabi-Yau warped $\QAC$-metrics }\label{cy.0}

To solve the complex Monge-Ampère equation \eqref{w.35}, we will first improve the decay of the Ricci potential using Corollary~\ref{mp.41}.  

\begin{lemma}
Assume that $n=3$ and $k=2$ or that $n\ge 4$ and $3\le k\le n-1$.   
  Then there exists $v\in \rho^{-\mu\left(  \frac{n(k-1)}{k(n-1)}\right)+2} w^{3-\mu} \CI_{\Qb}(W_{\epsilon,k})$ such that 
$$
          \frac{\sqrt{-1}}2 \pa \db \left( u_{\epsilon,k}+v \right)
$$ 
is the Kähler form of a warped $\QAC$-metric with Ricci potential 
$$
      r_v:= \log \left( \frac{\left( \frac{\sqrt{-1}}2\pa\db (u_{\epsilon,k}+v) \right)^n}{c_n\Omega_{\epsilon}^n \wedge \overline{\Omega}_{\epsilon}^n} \right)\in x_1^{4-2s} x_2^{2\mu} \CI_{\Qb}(W_{\epsilon,k})\subset
      \rho^{-2-\nu}\CI_{\Qb}(W_{\epsilon,k})
$$
for some $\nu>0$, where $s=0$ if $n\ge 4$ and $s>0$ is sufficiently small if instead $n=3$ and $k=2$.  
\label{cy.1}\end{lemma}
\begin{proof}
We will follow the strategy of \cite[Lemma~5.3]{CDR2016}.  Let $g_{\epsilon,k}$ be the Kähler metric with Kähler form $\omega_{\epsilon,k}:=\frac{\sqrt{-1}}2 \pa \db u_{\epsilon,k}$ and let 
$$
r_{\epsilon,k} \in x_1x_2^{\mu}\CI_{\Qb}(W_{\epsilon,k})= \rho^{-\mu\frac{n(k-1)}{k(n-1)}} w^{1-\mu}\CI_{\Qb}(W_{\epsilon,k})
$$ 
be its Ricci potential as defined in \eqref{w.34}.  Denote by $\Delta= -g_{\epsilon,k}^{ij}\nabla_i\nabla_j$ the Laplacian of $g_{\epsilon,k}$.  For $k$ and $n$ as in the statement of the lemma and with $\mu$ chosen sufficiently close to $2n-2$ when $k=\frac{2n}3$, we see that we can take $\delta-2= -\mu\frac{n(k-1)}{k(n-1)}$ and $\tau-2= 1-\mu$ in Corollary~\ref{mp.41}, so that there exists a unique 
$$
v_1\in \rho^{\delta}w^{\tau}\CI_{\Qb}(W_{\epsilon,k})= x_1^{3-\frac{2k(n-1)}{n(k-1)}} x_2^{\mu- \left( \frac{2k(n-1)}{n(k-1)} \right)}\CI_{\Qb}(W_{\epsilon,k})
$$ 
such that
$$
       \Delta v_1= 2 r_{\epsilon,k}.
$$ 
Taking $\mu$ sufficiently close to $2n-2$ in the case $k=\frac{2n}3$, notice that for our choices of $\delta$ and $\tau$,  
\begin{equation}
\begin{gathered}
      3-2\frac{k(n-1)}{n(k-1)}=  \frac{(n+2)k-3n}{n(k-1)} >0  \\
      \mu- 2\left( \frac{k(n-1)}{n(k-1)} \right)=\frac{\mu n(k-1)-2k(n-1) }{n(k-1)} >0,
\end{gathered}      
\end{equation}
so that $v_1$ decays at infinity.  Hence, we see that $\omega_{\epsilon,k} + \sqrt{-1}\pa\db v_1$ is positive definite outside a compact set.  In fact, thanks to the decay of $v_1$ at infinity, we can if needbe truncate $v_1$ outside a large compact set and assume without loss of generality that $\omega_1:=\omega_{\epsilon,k}+ \sqrt{-1}\pa\db v_1$ is positive definite everywhere on $W_{\epsilon,k}$.  Now, we compute that
\begin{equation}
(\omega_1)^n= (\omega_{\epsilon,k}+ \sqrt{-1}\pa\db v_1)^n= (1-\frac12 \Delta v_1)\omega_{\epsilon,k}^n + \frac{n!}{2!(n-2)!}\omega^{n-2}(\sqrt{-1}\pa\db v_1)^2+ \cdots + (\sqrt{-1}\pa\db v_1)^n,
\label{dev.1}\end{equation}
so that 
$$
   \omega_1^n- (1-r_{\epsilon,k})\omega_{\epsilon,k}^n\in (x_1 x_2^{\mu})^{2}\CI_{\Qb}(W_{\epsilon,k}).
$$
Hence, the Ricci potential of $\omega_1$ is such that 
$$
   r_1:= \log \left( \frac{\omega_1^n}{c_n\Omega_{\epsilon}^n\wedge \overline{\Omega}_{\epsilon}^n} \right) \in x_1^2 x_2^{2\mu}\CI_{\Qb}(W_{\epsilon,k})\subset x_1^2x_2^{\mu}\CI_{\Qb}(W_{\epsilon,k}).
$$
In particular, we have a better rate of decay for $r_1$.  Assuming for the moment that $4\le n\le 8$ when $k=3$, we can repeat this argument, this time however with $\omega_1$ and $r_1$ instead of $\omega_{\epsilon,k}$ and $r_{\epsilon,k}$, taking again $\delta= 2-\mu\frac{n(k-1)}{k(n-1)}$ in Corollary~\ref{mp.41}, but now with $\tau=4-\mu-s$ with $s=0$ if $n\ge 4$ and $s>(2n-2)-\mu\ge0$ sufficiently small when $n=3$ and $k=2$,  to find  $v_2\in \rho^{\delta}w^{\tau}\CI_{\Qb}(W_{\epsilon,k})$ such that 
$\omega_2 = \omega_1+\sqrt{-1}\pa\db v_2$ is positive definite with Ricci potential $r_2$ such that
$$
 r_2:=\log \left( \frac{\omega_1^n}{c_n\Omega_{\epsilon}^n\wedge \overline{\Omega}_{\epsilon}^n} \right) \in x_1^{4-2s} x_2^{2\mu}\CI_{\Qb}(W_{\epsilon,k}) \subset  x^{4-2s}\CI_{\Qb}( W_{\epsilon,k})= \rho^{-(4-2s)\frac{n(k-1)}{k(n-1)}}\CI_{\Qb}(W_{\epsilon,k}).
 $$
Now, since $k\ge 2$, we see that 
\begin{equation}
   (4-2s)\left( \frac{n(k-1)}{k(n-1)}\right)>2 \quad \Longleftrightarrow \quad  (n+1)k>2n + sn(k-1),
\label{ineq.1}\end{equation}
which holds if $s=0$ or $s>0$ is chosen small enough.  In either case, we see that $r_2\in \rho^{-2-\nu}\CI_{\Qb}(W_{\epsilon,k})$ for some  $\nu>0$.  Thus, it suffices to take
$v= 2(v_1+v_2)$.

If instead $k=3$ and $n\ge 9$, we can instead take $\delta-2= -2\mu\frac{n(k-1)}{k(n-1)}$ and $\tau=4-2\mu$ to find $v_2\in \rho^{\delta}w^{\tau}\CI_{\Qb}(W_{\epsilon,k})$ such that 
$\omega_2 = \omega_1+\sqrt{-1}\pa\db v_2$ is positive definite with Ricci potential $r_2$ such that
$$
 r_2:=\log \left( \frac{\omega_1^n}{c_n\Omega_{\epsilon}^n\wedge \overline{\Omega}_{\epsilon}^n} \right) \in x_1^{4} x_2^{4\mu}\CI_{\Qb}(W_{\epsilon,k}) \subset  x^{4}\CI_{\Qb}( W_{\epsilon,k})= \rho^{-4\frac{n(k-1)}{k(n-1)}}\CI_{\Qb}(W_{\epsilon,k}),
 $$
 so thanks to \eqref{ineq.1}, we can take $v= 2(v_1+v_2)$ again.  
\end{proof}

The Ricci potential of the Kähler metric in Lemma~\ref{cy.1} decays fast enough to apply the result of Tian-Yau \cite{Tian-Yau1991} or its parabolic version \cite{Chau-Tam2011}.  

\begin{theorem}
Assume that $n=3$ and $k=2$ or that $n\ge 4$ and $3\le k\le n-1$.  
Then there exists a unique solution  $\widetilde{v}\in \rho^{2-\mu\frac{n(k-1)}{k(n-1)}}w^{3-\mu}\CI_{\Qb}(W_{\epsilon,k})$ to the complex Monge-Ampère equation
$$
  \log\left(  \frac{ (\omega_{\epsilon,k}+ \sqrt{-1}\pa\db \widetilde{v})^n }{c_n\Omega_{\epsilon}^n\wedge \Omega_{\epsilon}^n} \right)= -r_{\epsilon,k}.
$$
In particular, $\tomega:= \omega_{\epsilon,k}+ \sqrt{-1}\pa\db \widetilde{v}$ is the Kähler form of a Calabi-Yau warped $\QAC$-metric on $W_{\epsilon,k}$.  
\label{cy.3}\end{theorem}
\begin{proof}
The proof is similar to \cite[Theorem~5.4]{CDR2016}, so we will go over the argument putting emphasis on the new features. The strategy is to apply the continuity method to the complex Monge-Ampère equation
\begin{equation}
\log\left(  \frac{ \left(\frac{\sqrt{-1}}{2}\pa\db(u_{\epsilon,k}+v +u)\right)^n }{c_n\Omega_{\epsilon}^n\wedge \Omega_{\epsilon}^n} \right)= -tr_v
\label{cy.5}\end{equation}
for $t\in [0,1]$, where $v$ is as in Lemma~\ref{cy.1}.  More precisely, we will show that the set 
$$
    \mathcal{S}=\{ s\in [0,1] \; | \; \exists u_s \in \rho^{2-\mu\frac{n(k-1)}{k(n-1)}}w^{3-\mu}\CI_{\Qb}(W_{\epsilon,k}) \; \mbox{a solution to \eqref{cy.5} for} \; t=s\}
$$
is in fact all of $[0,1]$ by showing that it is non-empty, open and closed.  Clearly, $u_0=0$ is a solution to \eqref{cy.5} for $t=0$, so that $\mathcal{S}$ is non-empty.  The openness of $\mathcal{S}$ follows from Corollary~\ref{mp.41}.  For closedness, suppose that $[0,\tau)\subset \mathcal{S}$ for some $0<\tau\le 1$.  Then we need to show that \eqref{cy.5} has a solution for $t=\tau$.  This can be done by deriving good a priori estimates on the solutions to \eqref{cy.5}.  First, thanks to Corollary~\ref{mp.34}, the Sobolev inequality holds for warped $\QAC$-metrics on $W_{\epsilon,k}$, so we can apply a Moser iteration to derive an a priori $\cC^0$-bound on solutions of \eqref{cy.5}.  The argument of Yau then gives uniform bounds on $\sqrt{-1}\pa\db v_t$, which by the result of Evans-Krylov, yields an a priori $\cC^{2,\gamma}_{w}(W_{\epsilon,k})$ bound on solutions.  Hence, taking an increasing sequence $t_i\nearrow \tau$, we can use the Arzela-Ascoli theorem to extract a subsequence of $\{u_{t_i}\}$ converging in $\cC^2_{w}(W_{\epsilon,k})$ to some solution $u_\tau$ of \eqref{cy.5} for $t=\tau$.  Bootstrapping, we see that in fact $u_{\tau}\in \CI_w(W_{\epsilon,k})$.

To see that this solution is in fact in $\rho^{2-\mu\frac{n(k-1)}{k(n-1)}}w^{4-\mu-s}\CI_{\Qb}(W_{\epsilon,k})$, we can first apply a Moser iteration with weights as in \cite[\S8.6.2]{Joyce} to obtain an a priori bound in $\rho^{-\nu_1}\cC^0_{w}(W_{\epsilon,k})$ for some $0<\nu_1<\nu$ with $\nu$ as in the statement of Lemma~\ref{cy.1}.  Appealing to Proposition~\ref{gb.12}, we use the Schauder estimate  in terms of warped $\QAC$-metrics to bootstrap and obtain that in fact $u_{\tau}\in \rho^{-\nu_1}\CI_w(W_{\epsilon,k})$.  Now, by Lemma~\ref{inc.1}, we see that 
$$
      \rho^{-\nu_1}\cC^1_w(W_{\epsilon,k})= (\chi w)^{\nu_1}\cC^1_w(W_{\epsilon,k})\subset \chi^{\nu_1}\cC^1_w(W_{\epsilon,k})\subset \cC^{0,\nu_1}_{\Qb}(W_{\epsilon,k}),
$$
which implies that $\| \pa\db u_{\tau}\|_{g_{\epsilon,k}}\in \cC^{0,\nu_1}_{\Qb}(W_{\epsilon,k})$.  Hence, thanks to Lemma~\ref{qb.6}, we can use the Schauder estimate in terms of $\Qb$-metrics to bootstrap and obtain that in fact
$u_\tau\in \rho^{-\nu_1}\CI_{\Qb}(W_{\epsilon,k})$.  Finally, using  Corollary~\ref{mp.41}, we conclude that $u_\tau$ is a solution in $\rho^{2-\mu\frac{n(k-1)}{k(n-1)}}w^{4-\mu-s}\CI_{\Qb}(W_{\epsilon,k})$, showing that $\mathcal{S}$ is closed, and hence is all of $[0,1]$.  

To show that the solution is unique, we can proceed as in \cite[Proposition~{7.13}]{Aubin}, but using Corollary~\ref{mp.41} instead of the maximum principle.

\end{proof}

\begin{corollary}
Assume that $n=3$ and $k=2$ or that $n\ge 4$ and $3\le k\le n-1$.   Then there exists  a complete Calabi-Yau metric $g$ on $\bbC^n$ with Euclidean volume growth whose tangent cone at infinity is $(W_{0,k}, g_{W_{0,k}})=(V_{0,k}\times \bbC,  g_{V_{0,k}}\times g_{\bbC})$.  In particular, $g$ is not isometric to the Euclidean metric on $\bbC^n$.  
\label{cy.4}\end{corollary}
\begin{proof}
It suffices to take the Calabi-Yau warped $\QAC$-metric of Theorem~\ref{cy.3} on $W_{\epsilon,k}\cong \bbC^n$.  
\end{proof}

\section{The case $k=2$ and $n\ge 4$} \label{k2.0}

When $k=2$ and $n\ge 4$, the decay of the Ricci potential is not sufficient to apply Corollary~5.19.  However, inspired by \cite[Proposition~12]{Szekelyhidi}, by a more direct means, we can first improve the decay of the Ricci potential at $\cH_2$.  Indeed, near that face, the metric associated to the potential $u_{\epsilon,k}$ of Corollary~\ref{w.31} has the same asymptotic behavior as the singular metric $g_{W_{0,k}}$.  This latter metric can be thought of as an $\AC$-metric,
\begin{equation}
   g_{W_{0,k}}= dR^2 + R^2 g_{2},  \quad \mbox{with } R= \sqrt{|z_{n+1}|^2+ \| z \|^{\frac{2(n-k)}{n-1}}_{N^*_{F_k}} } \asymp \rho,
\label{k2.1}\end{equation}
where 
\begin{equation}
   g_2= d\varphi^2+ (\sin\varphi)^2 g_{\pa\bV_{0,k}}+ (\cos \varphi)^2 g_{S_1}, \quad \varphi\in [0,\frac{\pi}2],
\label{k2.2}\end{equation}
is the spherical suspension of $g_{S_1}$, the standard metric on the unit circle, and $g_{\pa\bV_{0,k}}$, the metric on the cross-section $\pa\bV_{0,k}$ of the cone metric
$$
     g_{V_{0,k}}= dr^2+ r^2g_{\pa \bV_{0,k}}.
$$
There is in particular an incomplete edge singularity (or wedge singularity in the terminology of \cite{GKM}) at $\varphi=0$ corresponding to the boundary of $\cH_2\cap \cW_{\epsilon,k}$.  At $\varphi=\frac{\pi}2$, the metric appears singular, but is in fact smooth since $d\varphi^2+ (\cos\varphi)^2g_{S_1}$ is just the standard metric on $\bbS^2$ written in spherical coordinates.  Notice in particular that near $\cH_1\cap \cH_2\cap\cW_{\epsilon,k}$, we can if we wish use $\varphi$ instead of $x_1$ as a boundary defining function of $\pa \cH_2\cap\cW_{\epsilon,k}$.  Forgetting for the moment the singularities of the incomplete edge metric $g_2$, we can thus regard $g_{W_{0,k}}$ as an $\AC$-metric and a simple computation shows that its Laplacian is given by
\begin{equation}
\Delta_{g_{W_{0,k}}}=  R^{-2}\left( -\left( R\frac{\pa}{\pa R} \right)^2- (2n-2)R\frac{\pa}{\pa R} + \Delta_{g_2} \right),
\label{k2.3}\end{equation}
where $\Delta_{g_2}$ is the Laplacian of the metric $g_2$ on $\cH_2\cap \cW_{\epsilon,k}$.  Making the change of variable $X= R^{-1}$, this can be rewritten as
$$
      \Delta_{g_{W_{0,k}}}= X^2 B
$$
where 
\begin{equation}
B:= X^{-2}\Delta_{g_{W_{0,k}}}= -\left( X\frac{\pa}{\pa X} \right) + (2n-2)X\frac{\pa}{\pa X} + \Delta_{g_2}
\label{k2.4}\end{equation}
is a $b$-operator in the sense of Melrose \cite{MelroseAPS}, except for the fact that $g_2$ has an incomplete edge  singularity.  The indicial family of $B$, as defined in \cite[Definition~2.18]{MazzeoEdge}, is thus
\begin{equation}
I(B,\lambda)= -\lambda^2+ (2n-2)\lambda+ \Delta_{g_2},  \quad \lambda\in \bbC.  
\label{k2.5}\end{equation}
Since the Ricci potential \eqref{w.34} has a polyhomogeneous expansion at $\cH_2$, this suggests we can use the strategy of \cite[Lemma~5.44]{MelroseAPS} to eliminate part of the asymptotic expansion of $r_{\epsilon,k}$ at $\cH_2$ and construct a Kähler warped $\QAC$-metric having Ricci potential decaying faster at $\cH_2$.  However, since $\Delta_{g_2}$ is the Laplacian of an incomplete edge metric, we need to use the results of Appendix~\ref{ap.0}  about the mapping properties of $\Delta_{g_2}$.
  \begin{lemma}
For $k=2$ and $n\ge 4$, there exists $v\in \rho^2\cA_{\phg}(\cW_{\epsilon,k})$ contained in   $x_1^{-\frac{1}{n-2}}\rho^{2-\frac{n}{n-2}+\delta}\cA_{\phg}(\cW_{\epsilon,k})$ for all $\delta>0$ such that $\frac{\sqrt{-1}}{2}\pa\db (u_{\epsilon,k}+v)$ is the Kähler form of a warped $\QAC$-metric with Ricci potential
$$
      r_v:= \log \left( \frac{\left( \frac{\sqrt{-1}}2\pa\db (u_{\epsilon,k}+v) \right)^n}{c_n\Omega_{\epsilon}^n \wedge \overline{\Omega}_{\epsilon}^n} \right)\in x_1^{\nu} x_2^{\frac{4(n-1)}{n}+\nu}\cA_{\phg}(M)
$$
for some $\nu>0$. 
\label{k2.6}\end{lemma}
 
\begin{proof}
Recall first that with our assumption, $\mu= \frac{2(n-1)}{n-2}$ and $\rho=x^{-\frac{2(n-1)}{n}}$.
The idea is to take advantage of the fact that $r_{\epsilon,k}\in x_1x_2^{\mu}\cA_{\phg}(\cW_{\epsilon,k})$ has a polyhomogeneous expansion at $\cH_2\cap \cW_{\epsilon,k}$ and follow the strategy of \cite[Lemma~5.44]{MelroseAPS}.  Since 
$$
  x_1 x_2^{\mu}= x_1^{1-\mu} x^{\mu}= x_1^{1-\mu} \rho^{-\frac{n}{n-2}} \asymp x_1^{1-\mu} X^{\frac{n}{n-2}},
$$
we see that the top order term in the asymptotic expansion of $r_{\epsilon,k}$ is of the form $e_{\frac{n}{n-2}}X^{\frac{n}{n-2}}$ with
$e_{\frac{n}{n-2}}\in x_1^{1-\mu}\cA_{\phg}(\cH_2\cap\cW_{\epsilon,k})=x_1^{-1-\frac{2}{n-2}}\cA_{\phg}(\cH_2\cap\cW_{\epsilon,k})$.  To eliminate it, we need to find $f_{\frac{n}{n-2}}$ solving the equation
\begin{equation}
   I(B,\lambda)f_{\lambda}= e_{\frac{n}{n-2}} \quad \mbox{for} \quad  \lambda= -2+ \frac{n}{n-2}= -1+ \frac{2}{n-2}.
\label{k2.7}\end{equation}
Since $$
x_1^{-1-\frac{2}{n-2}}\cA_{\phg}(\cH_2\cap \cW_{\epsilon,k})\subset x_1^{-1-\frac{2}{n-2}}\cC_e^{\infty}(\cH_2\cap \cW_{\epsilon,k})\subset x_1^{-2-\frac{1}{n-2}}\cC_e^{\infty}(\cH_2\cap \cW_{\epsilon,k}),
$$
where $\CI_e(\cH_2\cap \cW_{\epsilon,k})= \cC^{\infty}_{x_1^{-2}g_2}(cH_2\cap \cW_{\epsilon,k})$ for the edge metric $x_1^{-2}g_2$, this suggests to take $a=-\frac{1}{n-2}$ in Corollary~\ref{ap.10}.  Suppose first that $e_{\frac{n}{n-2}}$ is in the image $I(B, -1+\frac{2}{n-2})$. Then there exists $f_{\frac{n}{n-2}}\in x_1^{-\frac{1}{n-2}}\cC^{2,\alpha}_2(\cH_2\cap \cW_{\epsilon,k})$ such that 
$$
      I(B,-1+\frac{2}{n-2})f_{\frac{n}{n-2}}= 2e_{\frac{n}{n-2}}.
$$
Moreover, by Corollary~\ref{ap.9}, we know that in fact $f_{\frac{n}{n-2}}\in x_1^{-\frac{1}{n-2}}\cA_{\phg}(\cH_2\cap \cW_{\epsilon,k})$. Hence extending $f_{\frac{n}{n-2}}$ smoothly away from $\cH_2\cap \cW_{\epsilon,k}$, we see that 
\begin{equation}
      \left(\Delta_{g_{W_{\epsilon,k}}}\left( \rho^{2- \frac{n}{n-2}}f_{\frac{n}{n-2}}\right)-  2r_{\epsilon,k}\right) \in  x_1^{-2-\frac{1}{n-2}}\rho^{-\frac{n}{n-2}-\delta}\cA_{\phg}(\cH_2\cap \cW_{\epsilon,k})= x_1^{\frac{2(n-1)\delta}{n}+\frac{1}{n-2}} x_{2}^{2\frac{n-1}n \left( \frac{n}{n-2}+\delta \right)} \cA_{\phg}(\cH_2\cap \cW_{\epsilon,k})
\label{k2.8}\end{equation}
for some $\delta>0$ depending on the polyhomogeneous expansion of $g_{W_{\epsilon,k}}$ and  $r_{\epsilon,k}$ at $\cH_2$.  If instead $e_{\frac{n}{n-2}}$ is not in the image of 
$I(B,-1+ \frac{2}{n-2})$, then proceeding as in the proof of \cite[Lemma~5.44]{MelroseAPS}, we can find $f_{\frac{n}{n-2}}, f_{(\frac{n}{n-2},1)}\in x_1^{-\frac{1}{n-2}}\cA_{\phg}(\cH_2\cap \cW_{\epsilon,k})$  with $f_{(\frac{n}{n-2},1)}$ in the kernel of $I(B,-1+ \frac{2}{n-2})$ such that 
\begin{equation}
 \left(\Delta_{g_{W_{\epsilon,k}}}\left( \rho^{2- \frac{n}{n-2}}\left(  f_{\frac{n}{n-2}} + f_{(\frac{n}{n-2},1)}\log \rho\right)  \right)-  2r_{\epsilon,k}\right) \in   x_1^{\frac{2(n-1)\delta}{n}+\frac{1}{n-2}} x_{2}^{2\frac{n-1}n \left( \frac{n}{n-2}+\delta \right)} \cA_{\phg}(\cH_2\cap \cW_{\epsilon,k})
 \label{k2.9}\end{equation}
for $\delta>0$ as above.
Since 
$$
    \frac{\rho^{2- \frac{n}{n-2}}f_{\frac{n}{n-2}}}{\rho^2}= \rho^{-\frac{n}{n-2}}f_{\frac{n}{n-2}} \in x_1^{2+\frac{1}{n-2}}x_2^{2\frac{n-1}{n-2}}\cA_{\phg}(\cH_2\cap \cW_{\epsilon,k})
$$
and $u_{\epsilon,k}\asymp \rho^2$ at infinity, we see also that, if necessary, we can change the definition of $f_{\frac{n}{n-2}}$ and $f_{(\frac{n}{n-2},1)}$ by truncating them using cut-off functions near $\cH_2$ in such a way that 
\begin{equation}
       \frac{\sqrt{-1}}{2} \pa\db \left( u_{\epsilon,k}+ \rho^{2-\frac{n}{n-2}}\left(f_{\frac{n}{n-2}} + f_{(\frac{n}{n-2},1)}\log \rho \right) \right)
\label{k2.10}\end{equation}
is still the Kähler form of a warped $\QAC$-metric.  Assuming without loss of generality that $\frac{2(n-1)\delta}{n}\le \frac{1}{n-2}$, we see by \eqref{k2.8} and  the expansion \eqref{dev.1} that  the Ricci potential $r_{\frac{n}{n-2}}$ of this new metric is such that
$$
       r_{\frac{n}{n-2}}\in x_1^{\frac{2(n-1)\delta}{n}+\frac{1}{n-2}} x_{2}^{\mu+2\frac{(n-1)\delta}n} \cA_{\phg}(\cH_2\cap \cW_{\epsilon,k}) = x_1^{-2-\frac{1}{n-2}}\rho^{-\frac{n}{n-2}-\delta}\cA_{\phg}(\cH_2\cap \cW_{\epsilon,k}).
$$
Looking at the top order term of the expansion of $r_{\frac{n}{n-2}}$ at $\cH_2\cap \cW_{\epsilon,k}$, we can again eliminate it using the same technique.  Indeed, if it is of the form 
$e_{\frac{n}{n-2}+\delta} X^{\frac{n}{n-2}+\delta}$ with $e_{\frac{n}{n-2}+\delta}  \in x_1^{-2-\frac{1}{n-2}}\cA_{\phg}(\cH_2\cap\cW_{\epsilon,k})$, then we can apply Corollary~\ref{ap.10} with $a= \frac{-1}{n-2}$ to remove this term by adding another correction to the Kähler metric \eqref{k2.10}, resulting in a Kähler warped $\QAC$-metric with Ricci potential in $x_1^{\frac{2(n-1)\delta'}{n}+\frac{1}{n-2}} x_{2}^{\mu+2\frac{(n-1)\delta'}n} \cA_{\phg}(\cH_2\cap \cW_{\epsilon,k})$ for some $\delta'>\delta$.  

Notice however that the top order term of $r_{\frac{n}{n-2}}$ at $\cH_2 \cap \cW_{\epsilon,k}$ may involve $(\log x_2)^{\ell}$ for some $\ell\in\bbN$, and similarly the top order term at the corner $\cH_1\cap \cH_2\cap \cW_{\epsilon,k}$ could involve some power of $\log x_1$ as well.  If $(\log x_2)^{\ell}$ does appear, then the top order term is of the form
$e_{\frac{n}{n-2}+\eta} X^{\frac{n}{n-2}+\eta}(\log X)^{\ell}$ with  $\eta>\delta$ and $e_{\frac{n}{n-2}+\eta}$ is polyhomogeneous on $\cH_2\cap \cW_{\epsilon,k}$ with
top order term at $\pa \cH_2\cap \cW_{\epsilon,k}$ of the form $b x_1^{-2-\frac{1}{n-2}-\eta'}(\log x_1)^{\ell'}$ for some $0\le \eta'\le \frac{2(n-1)(\eta-\delta)}{n}$ and $\ell'\ge \ell$, where $b$ is a smooth function on $\pa \cH_2\cap \cW_{\epsilon,k}$.    
Thus, it suffices in this case to apply Corollary~\ref{ap.10} with $a< -\frac{2}{n-2}-\eta'$ and handle the logarithmic terms as in \cite[Lemma~{5.44}]{MelroseAPS}.  

However, since we need $a> 5-2n$ to apply Corollary~\ref{ap.10}, we cannot take $\eta'$ arbitrarily large.  This means this method will not allow us to continue indefinitely and remove all terms in the asymptotic expansion of the Ricci potential.  What is important is that since 
$x_2^{\frac{4(n-1)}{n}}= x_1^{-\frac{4(n-1)}{n}}\rho^2$, we need to apply Corollary~\ref{ap.10} only when $a-2\ge -\frac{4(n-1)}{n}> 3-2n$ to eliminate the asymptotic expansion up to order $2$ in $\rho$.  Thus, we can continue this argument until we reach the desired decay of the Ricci potential.

\end{proof}

With this modification, we can now proceed as before to construct a Calabi-Yau warped $\QAC$-metric.  

\begin{theorem}
Assume that $k=2$ and $n\ge 4$.  
Then for $\delta>0$ arbitrarily small, there exists a solution  $\widetilde{v}\in x_1^{-\frac{1}{n-2}}\rho^{2-\frac{n}{n-2}+\delta}\CI_{\Qb}(W_{\epsilon,k})$ to the complex Monge-Ampère equation
$$
  \log\left(  \frac{ (\omega_{\epsilon,k}+ \sqrt{-1}\pa\db \widetilde{v})^n }{c_n\Omega_{\epsilon}^n\wedge \Omega_{\epsilon}^n} \right)= -r_{\epsilon,k}.
$$
In particular, $\tomega:= \omega_{\epsilon,k}+ \sqrt{-1}\pa\db \widetilde{v}$ is the Kähler form of a Calabi-Yau warped $\QAC$-metric on $W_{\epsilon,k}$.  \label{k2.11}\end{theorem}
\begin{proof}
Using the function $v$ of Lemma~\ref{k2.6}, we need to solve the Monge-Ampère equation
$$
\log\left(  \frac{ (\omega_{\epsilon,k}+ \sqrt{-1}\pa\db (v+u))^n }{c_n\Omega_{\epsilon}^n\wedge \Omega_{\epsilon}^n} \right)= -r_{\epsilon,k}.
$$
Proceeding as in the proof of Lemma~\ref{cy.1}, we can first improve the decay of the Ricci potential by applying Corollary~\ref{mp.41} with $\tau-2= -\frac{4(n-1)}{n}$ and 
$\delta = - \frac{n\nu}{2(n-1)}$ to get instead a Ricci potential in 
$$
       x_1^{2\nu}x_2^{\frac{8(n-1)}{n}+2\nu}\CI_{\Qb}(W_{\epsilon,k})\subset x_1^{2\nu}x_2^{\frac{4(n-1)}{n}+2\nu}\CI_{\Qb}(W_{\epsilon,k}).
$$
We can then reapply Corollary~\ref{mp.41}, always with $\tau-2= -\frac{4(n-1)}n$, but now with $\delta=- 2\frac{n\nu}{2(n-1)}$, and then $\delta=- 4\frac{n\nu}{2(n-1)}$, and then
$ \delta=- 8\frac{n\nu}{2(n-1)}$ and so forth until we obtain a Ricci potential in
$$
x_1^{2^\ell\nu}x_2^{\frac{2^{\ell}4(n-1)}{n}+2\nu}\CI_{\Qb}(W_{\epsilon,k})
$$
with $2^{\ell}\nu> \frac{4(n-1)}{n}$ so that the Ricci potential decays faster than quadratically at infinity.  We can then apply the continuity method as in the proof of Theorem~\ref{cy.3} to obtain the desired Calabi-Yau warped $\QAC$-metric.     
\end{proof}

\appendix

\section{Mapping properties of the scalar Laplacian of an incomplete edge metric} \label{ap.0}

The results presented in this appendix are known to experts, but they did not seem to appear in the literature in the explicit form that we needed.  

Let $M$ be a compact manifold of dimension $m$ with boundary equipped with fibre bundle structure
\begin{equation}
    \xymatrix{ Z \ar[r] & \pa M \ar[d]^{\varpi} \\
       & Y, 
    }
\label{ap.1}\end{equation}
where $Y$ and $Z$ are closed manifolds of dimension $\operatorname{b}$ and $\dv$ respectively.  Let $r\in \CI(M)$ be a boundary defining function.  On $M\setminus \pa M$, we consider an incomplete edge metric $g_{ie}$ which near $\pa M$ takes the form
\begin{equation}
     g_{ie}= dr^2 + \varpi^*g_Y + r^2 \kappa,
\label{ap.2}\end{equation}
where $g_Y$ is a metric on $Y$ and $\kappa$ is a symmetric $2$-tensor which restricts to a metric on each fibre of $\varpi: \pa M\to Y$, \cf \cite{Mazzeo-Vertman}.  We also require that $\varpi: \pa M \to Y $ is a Riemannian submersion with respect to the metrics $\varpi^* g_Y + \kappa$ and  $g_Y$.  
In particular, when we fix $y\in Y$, this induces a cone metric on $(0,\delta)_r\times \varpi^{-1}(y)$ for some $\delta>0$.  Notice that the conformal metric $g_e:= \frac{g_{ie}}{r^2}$ is an edge metric in the sense of \cite{MazzeoEdge}.   

Let $\Delta_{ie}:= -\operatorname{div}\circ \nabla$ be the Laplacian of $g_{ie}$.  Then $r^2\Delta_{ie}$ is an elliptic edge operator, so that the edge calculus of Mazzeo \cite{MazzeoEdge} can be used to determine the mapping properties of $\Delta_{ie}$.  Let $L^2_{ie}(M)$ be the space of square integrable functions with respect to the volume density of $g_{ie}$.  If $H^k_e(M)$ denotes the $L^2$-Sobolev space of order $k$ with respect to the edge metric $g_e$, then set also
\begin{equation}
   H^k_{ie}(M):= r^{-\frac{m}2}H^k_e(M),
\label{ap.3}\end{equation}   
where the factor $r^{-\frac{m}2}$ is there to ensure that $H^0_{ie}(M)=L^2_{ie}(M)$.  Notice however that $H^k_{ie}(M)$ is not the natural Sobolev space associated to $g_{ie}$. Indeed, we do use the volume density of $g_{ie}$, but derivatives are instead measured with respect to the edge metric $g_e$.  As explained in \cite{MazzeoEdge}, besides the principal symbol, we need to consider two model operators to determine the mapping properties of the edge operator $r^2\Delta_{ie}$.  We need first to consider its indicial operator, which for fixed $y\in Y$ is given by
\begin{equation}
 I_y(r^2\Delta_{ie})= -\left(r\frac{\pa}{\pa r}\right)^2- (\dv-1)r\frac{\pa}{\pa r} + \Delta_{\kappa_y},
\label{ap.4}\end{equation}         
where $\Delta_{\kappa_y}$ is the Laplacian induced by $\kappa$ on $\varpi^{-1}(y)$.  The corresponding indicial family is then given by
\begin{equation}
I_y(r^2\Delta_{ie},\lambda)= -\lambda^2-(\dv-1)\lambda + \Delta_{\kappa_y}.
\label{ap.5}\end{equation}
Since $\Delta_{\kappa_y}$ has positive spectrum, notice that this family is invertible for $1-\dv<\lambda<0$.  
One needs also to consider its normal operator \cite[Definition~2.16]{MazzeoEdge} defined for $y\in Y$ by
\begin{equation}
N_y(r^2\Delta_{ie})= -\left( r\frac{\pa}{\pa r} \right)^2 -(\dv-1)r\frac{\pa}{\pa r} + r^2 \Delta_{T_yY} + \Delta_{\kappa_y},
\label{ap.6}\end{equation}
where $\Delta_{T_yY}$ is the Euclidean Laplacian on $T_yY$ induced by the metric $g_Y$.  

\begin{lemma}
For $a> \frac{3-\dv}2$, the normal operator  $N_y(r^2\Delta_{ie})$ is injective when acting on  the weighted $L^2$-space $r^a L^2(\bbR^+_r\times T_yY\times \varpi^{-1}(y))$ defined in terms of the volume density of the metric 
$$
    dr^2+ g_{T_yY}+ \kappa_y,
$$
where $g_{T_yY}$ is the Euclidean metric obtained by restricting $g_Y$ to $T_yY$.  
\label{ap.7}\end{lemma}
\begin{proof}
We follow the approach of \cite[Lemma~5.5]{ALMP2012}.  Since 
$$
       L^2_{g_w}(\bbR^+_r\times T_yY\times \varpi^{-1}(y))= r^{-\frac{\dv-1}2}L^2_h(R^+_s\times T_yY\times \varpi^{-1}(y))
$$
with $h= r^2dr^2 + g_{T_yY}+ \kappa_y$, we can instead look at the conjugated operator 
$$
  r^{\frac{\dv-1}2}N_y(r^2\Delta_{ie})r^{-\frac{\dv-1}2}= -\left( r\frac{\pa}{\pa r} \right)^2  + \left( \frac{\dv-1}2\right)^2 + r^2 \Delta_{T_yY} + \Delta_{\kappa_y}
$$
acting on $r^aL^2_h(\bbR^+_r\times T_yY\times \varpi^{-1}(y))$. Taking the Fourier transform on $T_yY$ and denoting by $\xi$ the dual variable gives the operator
$$
\widehat{r^{\frac{\dv-1}2}N_y(r^2\Delta_{ie})r^{-\frac{\dv-1}2}}= -\left( r\frac{\pa}{\pa r} \right)^2  + \left( \frac{\dv-1}2\right)^2 + r^2 |\xi|^2 + \Delta_{\kappa_y}.
$$
Introducing the variables $t:=r|\xi|$ and $\widehat{\xi}:= \frac{\xi}{|\xi|}$, this can be rewritten as
$$
\widehat{r^{\frac{\dv-1}2}N_y(r^2\Delta_{ie})r^{-\frac{\dv-1}2}}= -\left( t\frac{\pa}{\pa t} \right)^2  + \left( \frac{\dv-1}2\right)^2 + t^2 + \Delta_{\kappa_y}.
$$
To describe the nullspace of this operator, we decompose it in terms of the spectrum of $\Delta_{\kappa_y}$.  Thus, restricting to the $\lambda$-eigenspace of 
$\Delta_{\kappa_y}$ gives 
$$
-\left( t\frac{\pa}{\pa t} \right)^2  + \left( \frac{\dv-1}2\right)^2 + t^2 + \lambda.
$$
The nullspace of this operator are the solutions to the modified Bessel equation, so is described in terms of the modified Bessel functions, 
$$
   A I_{\nu}(t)+  B K_{\nu}(t), \quad \nu= \sqrt{\lambda+ \left( \frac{\dv-1}2 \right)^2}.
$$
The function $I_{\nu}$ increases exponentially as $t\nearrow \infty$ so we must have $A=0$ for solutions to be in 
$$
r^aL^2_h(\bbR^+_r\times  T_yY\times  \varpi^{-1}(y)).
$$  
On the other hand, $K_{\nu}$ decreases exponentially as $t\nearrow \infty$, whereas as $t\searrow 0$, 
$$
     K_{\nu}\asymp \left\{ \begin{array}{ll} t^{-|\nu|}, & \mbox{if} \; \nu\ne 0, \\ \log t, & \mbox{if} \; \nu=0.  \end{array}  \right.
$$
Thus, to have an injective operator when acting on $r^aL^2_h(\bbR^+_r\times T_yY\times \varpi^{-1}(y))$, we need that $2|\nu|+ 2a\ge 2$, that is, $a\ge 1-\sqrt{\lambda+\left(\frac{\dv-1}2\right)^2}$.  Since $\lambda\ge 0$, the result follows.  

\end{proof}

\begin{proposition}
If $\dv> 3$, the operator $\Delta_{ie}$ is essentially self-adjoint with self-adjoint extension given by
$$
     \Delta_{ie}: r^2 H^2_{ie}(M) \to L^2_{ie}(M).
$$
\label{ap.10}\end{proposition}
\begin{proof}
By \cite[Lemma~2.2]{Mazzeo-Vertman} and the fact that the indicial family \eqref{ap.5} is invertible for $1-\dv< \lambda <0$, we see that $\Delta_{ie}$ is essentially self-adjoint.
On the other hand, since the indicial family \eqref{ap.5} is invertible for $\lambda= \frac{-\dv+3}2$, we know from \cite[Theorem~4.2]{GKM} that the minimal extension is $r^2 H^2_{ie}(M)$.
\end{proof}

\begin{theorem}
For $\frac{3-\dv}2< a< \frac{\dv+1}2$ and $\ell\in \bbZ$,  the operator 
\begin{equation}
      \Delta_{ie} : r^a H^{\ell+2}_{ie}(M)\to r^{a-2}H^{\ell}_{ie}(M)
\label{ap.8a}\end{equation}
is Fredholm and has discrete spectrum.  Moreover, if we assume that the spectrum of the Laplacian $\Delta_{\kappa_y}$ does not depend on the choice of $Y$, then the inverse $G$ is in  $\Psi^{-2,\cH}_e(M)$ and is such that 
$$
      G(r^2 \Delta_{ie})= \Id - P_1,  \quad r^2 \Delta_{ie}G= \Id -P_2
$$  
with $P_1\in \Psi^{-\infty,\cE'}(M)$ and  $P_2\in \Psi^{-\infty,\cF'}(M)$  very residual operators, where $\cH$, $\cE'$ and $\cF'$ are index families as described in 
\cite[Theorem~6.1]{MazzeoEdge}.    
\label{ap.8}\end{theorem}
\begin{proof}
By Lemma~\ref{ap.7}, \cite[Theorem~6.1]{MazzeoEdge}\footnote{See also \cite[Proposition~4.2 and p.3260]{HM2005}  for the case of variable eigenvalues for the Laplacian $\Delta_{\kappa_y}$.} and the fact that the indicial family \eqref{ap.5} is invertible for $1-\dv<\lambda<0$, we know that the map \eqref{ap.8a} is essentially injective, that is, has closed range and finite dimensional kernel.  Since $\Delta_{ie}$ is formally self-adjoint with respect to the metric $g_{ie}$, the adjoint of the map \eqref{ap.8a} is 
$$
  \Delta_{ie}: r^{2-a}H^{-\ell}_{ie}(M)\to r^{-a}H^{-\ell-2}_{ie}(M),
$$
which by the same argument is essentially injective.  This means that the map \eqref{ap.8a} in in fact Fredholm.  The left and right inverse are then given by \cite[Theorem~6.1]{MazzeoEdge}.  Since $r^a H^{\ell+2}_{ie}(M)$ is compactly included in $r^{a-2} H^{\ell}_{ie}(M)$, the left inverse $Gr^2$ is compact, so the spectrum of $\Delta_{ie}$ must be discrete.   
\end{proof}

\begin{corollary}
Assume that the spectrum of $\Delta_{\kappa_y}$ does not depend on $y\in Y$ and that $a>\frac{3-\dv}2$.  In this case,  if $u\in r^a L^2_{ie}(M)$ is such that $\Delta_{ie}u$ is polyhomogeneous, then $u$ itself is polyhomogeneous.  Similarly, the eigenfunctions of $\Delta_{ie}$ contained in $r^a H^2_{ie}(M)$ are polyhomogeneous.
\label{ap.9}\end{corollary}
\begin{proof}
It suffices to proceed as  in \cite[Proposition~7.17]{MazzeoEdge} using that $P_2$ in Theorem~\ref{ap.8} is very residual.  For the statement about the eigenfunctions, notice that $r^2(\Delta_{ie}-\lambda)$ has the same principal symbol, the same indicial operator and the same normal operator as $r^2\Delta_{ie}$.  Hence, the conclusion of Theorem~\ref{ap.8} also holds for this operator, so elements in the kernel of $\Delta_{ie}-\lambda$ that are in $r^a L^2_{ie}(M)$ must be polyhomogeneous.  
\end{proof}

There is also a corresponding statement in terms of edge Hölder spaces, \cf \cite{Szekelyhidi}.  

\begin{corollary}
Assume that the spectrum of $\Delta_{\kappa_y}$ does not depend on $y\in Y$.  Then for $-\dv+1<a<0$, $\ell\in\bbN_0$ and $\alpha\in (0,1)$, the operator 
$$
     \Delta_{ie}: r^a\cC^{\ell+2,\alpha}_e(M) \to r^{a-2}\cC^{\ell,\alpha}_e(M)
$$
is Fredholm with discrete spectrum, where $\cC^{\ell,\alpha}_e(M)$ denotes the edge Hölder space of \cite{MazzeoEdge}. 
\label{ap.10}\end{corollary}
\begin{proof}
This is a consequence of Theorem~\ref{ap.8} and \cite[Corollary~5.4]{MazzeoEdge}.  The change of range of admissible values of $a$ is related to the fact that there is a natural inclusion $r^a\cC^{\ell,\alpha}_e(M)\subset r^{a+\frac{\dv+1}2}L^2_{ie}(M)$.  
\end{proof}

\bibliography{CYCn}
\bibliographystyle{amsplain}

\end{document}